\documentclass[reqno,11pt]{article}
\usepackage{a4wide}
\usepackage{color,eucal,enumerate,mathrsfs,amsthm}
\usepackage[normalem]{ulem}
\usepackage{amsmath,amssymb,amsthm,epsfig,bbm,makecell}
\numberwithin{equation}{section}

\usepackage{amsfonts}
\usepackage{dsfont}
\usepackage{mathtools}
\usepackage{leftidx}
\usepackage{graphicx}
\usepackage{esint}
\usepackage{authblk}
\usepackage{multirow}

\usepackage[pdfborder={0 0 0}]{hyperref}  

\usepackage[latin1]{inputenc}
\usepackage[english]{babel}

\newtheorem{Definition}{Definition}[section]
\newtheorem{Proposition}[Definition]{Proposition}
\newtheorem{Lemma}[Definition]{Lemma}
\newtheorem{Theorem}[Definition]{Theorem}
\newtheorem{Corollary}[Definition]{Corollary}
\theoremstyle{definition}

\allowdisplaybreaks

\newcommand{\B}{\mathbb{B}}

\newcommand{\N}{\mathbb{N}}

\newcommand{\R}{\mathbb{R}}

\newcommand{\mm}{{\mbox{\boldmath$m$}}}

\newcommand{\ggamma}{{\mbox{\boldmath$\gamma$}}}

\newcommand{\ppi}{{\mbox{\boldmath$\pi$}}}

\newcommand{\sfd}{{\sf d}}

\newcommand{\sfh}{{\sf h}}

\newcommand{\sfB}{{\sf B}}

\newcommand{\sfP}{{\sf P}}

\newcommand{\sfR}{{\sf R}}

\newcommand{\Kliminf}{K\kern-3pt-\kern-2pt\mathop{\rm lim\,inf}\limits}  
\newcommand{\supp}{\mathop{\rm supp}\nolimits}   
\newcommand{\Lip}{\mathop{\rm Lip}\nolimits}          
\renewcommand{\d}{{\mathrm d}}
\newcommand{\dt}{{\d t}}
\newcommand{\ddt}{{\frac \d\dt}}

\newcommand{\restr}[1]{\lower3pt\hbox{$|_{#1}$}} 

\newcommand{\la}{\left<}                  
\newcommand{\ra}{\right>}
\newcommand{\eps}{\varepsilon}  
\newcommand{\nchi}{{\raise.3ex\hbox{$\chi$}}}
\newcommand{\weakto}{\rightharpoonup}

\newcommand{\prob}[1]{\mathscr P(#1)}                   
\newcommand{\probt}[1]{\mathscr P_2(#1)}                   
\newcommand{\e}{{\rm{e}}}                           

\renewcommand{\mm}{\mathfrak m}                                

\renewenvironment{proof}{\removelastskip\par\medskip   
\noindent{\em Proof.} \rm}{\penalty-20\null\hfill$\square$\par\medbreak}

\newcommand{\testi}[1]{{\rm Test}^{\infty}(#1)}

\newcommand{\X}{{\rm X}}

\newcommand{\h}{{\sfh}}

\renewcommand{\ae}{{\textrm{\rm{-a.e.}}}}

\newcommand{\RCD}{{\sf RCD}}

\newcommand{\lip}{{\rm lip}}
\newcommand{\AC}{{\rm AC}}

\newcommand{\HS}{{\lower.3ex\hbox{\scriptsize{\sf HS}}}}

\newcommand{\Y}{{\rm Y}}
\newcommand{\hp}{{\sf p}}

\newcommand{\bs}{{\sf bs}}
\newcommand{\loc}{{\sf loc}}
\renewcommand{\B}{{\sf B}}

\title{Viscosity solutions of Hamilton--Jacobi equation in $\RCD(K,\infty)$ spaces and applications to large deviations}
\author{Nicola Gigli \thanks{SISSA, via Bonomea 265, 34136 Trieste, Italy. email: ngigli@sissa.it} \quad Luca Tamanini \thanks{Universit\`a Cattolica del Sacro Cuore, Dipartimento di Matematica e Fisica ``Niccol\`o Tartaglia'', via della Garzetta 48, 25133, Brescia, Italy. email: luca.tamanini@unicatt.it} \quad Dario Trevisan \thanks{Universit\`a di Pisa,  Dipartimento di Matematica, Largo Bruno Pontecorvo 5, 56127  Pisa Italy. email: dario.trevisan@unipi.it. Member of the GNAMPA INdAM group}}

\begin{document}

\maketitle

\begin{abstract}
\noindent The aim of this paper is twofold.
\begin{itemize} 
\item[-] In the setting of $\RCD(K,\infty)$ metric measure spaces, we derive uniform gradient and Laplacian contraction estimates along solutions of the viscous approximation of the Hamilton--Jacobi equation. We use these estimates to prove that, as the viscosity tends to zero, solutions of this equation converge to the evolution driven by the Hopf--Lax formula, in accordance with the smooth case.

\item[-] We then use such convergence to study the small-time Large Deviation Principle for both the heat kernel and the Brownian motion: we obtain the expected behavior under the additional assumption that the space is proper.
\end{itemize}

\noindent As an application of the latter point, we also discuss the $\Gamma$-convergence of the Schr\"odinger problem to the quadratic optimal transport problem in proper $\RCD(K,\infty)$ spaces.
\end{abstract}

\tableofcontents

\section{Introduction}

The heat kernel surely plays a crucial role in geometric and stochastic analysis and  understanding  its behavior is therefore of particular importance in applications and estimates. The celebrated Varadhan's formula \cite{Varadhan67} for the heat kernel $\hp_t(x,y)$, i.e.\ the fundamental solution of the heat equation $\partial_t\hp_t = \Delta_g\hp_t$ on a smooth Riemannian manifold, $\Delta_g$ being the Laplace-Beltrami operator, is a remarkable example in this sense, as it links analysis, probability and geometry. It states that, in the small-time regime, $\hp_t$ behaves in the following `Gaussian' way
\begin{equation}\label{eq:varadhan}
\lim_{t \downarrow 0} t\log\hp_t(x,y) = -\frac{\sfd^2_g(x,y)}{4}\,,
\end{equation}
where $\sfd_g$ is the Riemannian distance induced by $g$. The analytic/probabilistic information on the left-hand side is thus linked to the geometric one on the right-hand side. If one integrates the heat kernel over two positive-measure sets $U$ and $V$, then a Laplace-type argument allows to extend the pointwise validity of \eqref{eq:varadhan} to a setwise information, namely
\begin{equation}\label{eq:hino-ramirez}
\lim_{t \downarrow 0} t\log\sfP_t(U,V) = -\frac{\sfd^2(U,V)}{4}\,,
\end{equation}
where
\begin{equation}\label{eq:double-integral}
\sfP_t(U,V) := \int_U\int_V\hp_t(x,y)\,\d {\rm vol}(x)\d {\rm vol}(y)
\end{equation}
and $\sfd^2(U,V) := \inf\{\sfd^2_g(x,y) \,:\, x\in U,\,y \in V\}$. If either $U$ or $V$ shrinks down to a singleton, say $U=\{x\}$, then \eqref{eq:hino-ramirez} heuristically boils down to
\begin{equation}\label{eq:ldp-intro}
\lim_{t \downarrow 0} t\log\int_V \hp_t(x,y)\,\d {\rm vol}(y) = -\frac{\sfd^2(x,V)}{4}
\end{equation}
and in particular this strongly suggests that the measures $\mu_t[x] := \hp_t(x,y){\rm vol}(y)$ satisfy a \emph{Large Deviations Principle} (also called LDP for short) with speed $1/t$ and rate function $I(y) =\tfrac14 \sfd^2(x,y)$. Namely,
\[
\begin{split}
& \liminf_{t \downarrow 0} t\log\mu_t[x](O) \geq -\inf_{y \in O}\frac{\sfd^2(x,y)}{4} \qquad \forall O \subset \X \textrm{ open}, \\
& \limsup_{t \downarrow 0} t\log\mu_t[x](C) \leq -\inf_{y \in C}\frac{\sfd^2(x,y)}{4} \qquad \forall C \subset \X \textrm{ closed}.
\end{split}
\]
As the heat kernel $\hp_t(x,y)$ is the transition density of the Brownian motion, this couple of inequalities allows to quantify how rare the event is, that a diffusive particle starting at $x$ belongs to a certain set after a time $t$. For sure such an event is rare if $x \notin V$, since $\hp_t(x,\cdot){\rm vol} \rightharpoonup \delta_x$ as $t \downarrow 0$, but this is a sort of 0th order information coming from the law of large numbers. A Large Deviations Principle, if verified, is a much more accurate estimate, as it tells us that the decay is exponential (in the Euclidean setting, this would be ensured at a pointwise level by the Central Limit Theorem) and the description of the decay is made global via the rate function (because of this globality feature, LDP is more precise than Central Limit Theorem). 

And again, as the heat kernel $\hp_t(x,y)$ is the transition density of the Brownian motion, one may wonder whether the discussion carried out so far can be lifted to a higher level, namely if analogous estimates hold true for the law of Brownian motions with suitably vanishing quadratic variation. Focusing on \eqref{eq:ldp-intro} the answer is affirmative, as shown for instance in \cite{Azencott80, FreidlinWentzell98}. In order to guess what \eqref{eq:ldp-intro} should look like, little effort is needed for the left-hand side; a bit more is instead required to understand what should replace $\sfd^2(x,\cdot)$ as rate function. Heuristically though, the argument is very natural: in the same way the heat kernel is the building block of the Brownian motion (actually the density of the joint law at times 0 and $t$ of the Brownian path measure), so $\sfd^2(x,y)$ is the foundation of the kinetic energy, since for every absolutely continuous curve $\gamma : [0,1] \to \X$ it holds
\[
\int_0^1 |\dot{\gamma}_t|^2\,\d t = \sup\sum_{i=1}^n \frac{\sfd^2(\gamma_{t_i},\gamma_{t_{i-1}})}{t_i-t_{i-1}}\,,
\]
where the supremum runs over all finite partitions of $[0,1]$. Therefore, the lifting of \eqref{eq:ldp-intro} reads as
\begin{equation}
\label{eq:ldp-path}
\lim_{t \downarrow 0} t\log \B_t[x](\Gamma) = -\inf_{\gamma \in \Gamma \,:\, \gamma_0 = x}\frac{1}{4} \int_0^1 |\dot{\gamma}_t|^2\,\d t\,,
\end{equation}
where $\B_t[x]$ denotes the law on $C([0,1],\X)$ of the Brownian motion starting at $x$ with quadratic variation equal to $t$ and $\Gamma \subseteq C([0,1],\X)$.

\medskip

As long as $\X$ is a smooth Riemannian manifold with Ricci curvature bounded from below, it is not difficult to show that the three descriptions of the heat kernel small-time asymptotics \eqref{eq:varadhan} \eqref{eq:hino-ramirez} \eqref{eq:ldp-intro} presented above are equivalent (e.g.\ using the fact that, by Bishop--Gromov inequality, we know the behavior of the volume measure at small scales). However, when $\X$ is replaced by an infinite-dimensional (and possibly non-smooth) space, the relationship between the three is not clear at all. Actually, only some of them are known to hold and the same for \eqref{eq:ldp-path}. 

A rich literature has flourished around \eqref{eq:hino-ramirez}. In Hilbert spaces, hence still within a smooth framework, Zhang \cite{Zhang00} was the first to obtain a general and satisfactory result, until the breakthrough of Hino--Ram\'irez \cite{HinoRamirez03} and Ariyoshi--Hino \cite{AriyoshiHino05}, who discovered that the `Gaussian' behavior of $\hp_t$ is in fact a universal paradigm valid on general strongly local symmetric Dirichlet spaces if rephrased in the integrated form \eqref{eq:hino-ramirez} (for an appropriate set-distance associated with the given Dirichlet form and for transition measures rather than kernels, whose existence is not always granted). In this direction, it is worth mentioning the recent papers \cite{DelloSchiavoSuzuki20, DelloSchiavoSuzuki21, DelloSchiavoSuzuki21b} which refine (1.2) in the settings considered by the authors. As concerns the pointwise version \eqref{eq:varadhan}, instead, it has essentially been investigated only in finite-dimensional, locally compact contexts so far, as for instance Lipschitz \cite{Norris97} and sub-Riemannian manifolds \cite{Leandre87, Leandre87b}; it can also be obtained as byproduct of Gaussian heat kernel estimates in the case of finite-dimensional $\RCD$ spaces \cite{JLZ15} and locally compact Dirichlet spaces supporting local Poincar\'e and doubling \cite{Sturm96I, Sturm96II}. 

The current knowledge around \eqref{eq:ldp-intro} outside the Riemannian realm is closer to \eqref{eq:varadhan} rather than \eqref{eq:hino-ramirez}, in the sense that --except for the case of Gaussian spaces-- to the best of our knowledge there  are no results in genuinely infinite-dimensional spaces, i.e.\ without local compactness and local doubling. 

\medskip

\noindent\textbf{Main contributions}. In this paper we make an effort toward the understanding of LDP for heat kernels and Brownian motions on $\RCD(K,\infty)$ spaces. The main motivation comes from the central role played by the heat semigroup in conjunction with curvature bounds. This connection, already highlighted in \cite{GigliKuwadaOhta13} in the setting of Alexandrov spaces, is even stronger when it comes to lower Ricci bounds. Indeed, the heat flow is at the very heart of the definition of $\RCD$ spaces \cite{AmbrosioGigliSavare11-2, Gigli12} and is intimately linked to the geometry of such spaces, as it appears clearly e.g.\ in Bakry--\'Emery gradient estimate. The study of LDP for heat kernels and Brownian motions within the $\RCD$ setting thus appears a rather natural question to be addressed.

In what follows, for $x \in \X$ and  $t>0$ we set $\mu_t[x] := \hp_t(x,\cdot)\mm$.  Our first main result can be roughly stated as follows (see Proposition \ref{pro:from-hjb-to-hopflax} for the rigorous statement):
\begin{Theorem}\label{thm:1}
Let $(\X,\sfd,\mm)$ be an $\RCD(K,\infty)$ space with $K \in \R$. Then for every $\varphi:\X\to\R$  `smooth enough' we have
\begin{equation}
\label{eq:visc}
\lim_{\eps \downarrow 0}\eps\log\Big(\int e^{\varphi/\eps}\,\d\mu_\eps[x]\Big)= \sup_{y \in \X}\Big\{\varphi(y) - \frac{\sfd^2(x,y)}{4}\Big\}\qquad\forall  x\in\X.
\end{equation}
\end{Theorem}
Here `smooth enough' is interpreted in the sense of so-called \emph{test functions} $\testi\X$, see \cite{Savare13}, \cite{Gigli14}: for the purpose of this introduction it is only important to keep in mind that this is a rather large space of functions, meaning that it is dense in many relevant functional spaces.

By the monotonicity in $\varphi$ of both sides of \eqref{eq:visc} it is easy to see that the same formula holds for the larger class of functions $\psi:\X\to\R$ such that
\begin{equation}
\label{eq:classe}
\psi=\sup_{\substack{\varphi\in\testi\X \\ \varphi\leq \psi}}\varphi=\inf_{\substack{\varphi\in\testi\X \\ \varphi\geq \psi}}\varphi\,.
\end{equation}
If $\X$ is also proper, then it is known that the class of test functions contains many cut-off functions (see \cite[Lemma 6.7]{AmbrosioMondinoSavare13-2}), so that the class of $\psi$'s for which \eqref{eq:classe} holds coincides with $C_b(\X)$ (see Proposition \ref{lem:ldp-lipschitz}). Also, if $\X$ is proper, then for every $x\in\X$ the function $y\mapsto \frac{\sfd^2(x,y)}{4}$ has compact sublevels (in other words, $\X$ is locally compact; in fact this is a characterization of proper spaces, see \cite[Theorem 2.5.28]{BBI01}) and it is well known that in this case the validity of \eqref{eq:visc} for every $\varphi\in C_b(\X)$ is equivalent to a LDP, see Theorems \ref{thm:lower-ldp} and \ref{thm:upper-ldp} below, so we obtain:

\begin{Theorem}[Theorem \ref{thm:ldp-heatkernel}]\label{thm:main1}
Let $(\X,\sfd,\mm)$ be a proper $\RCD(K,\infty)$ space with $K \in \R$. Then the family $(\mu_t[x])$ satisfies a Large Deviations Principle with speed $t^{-1}$ and rate
\[
I(z) := \frac{\sfd^2(x,z)}{4}\,.
\]
\end{Theorem}
It is certainly natural to wonder whether the same result holds without assuming the space to be proper (in this case LDP should be interpreted in the so-called `weak' sense). We do not know: The main problem is the absence of a suitable localization procedure that allows to deduce that the class of $\psi$'s for which \eqref{eq:classe} holds coincides with $C_b(\X)$.

\bigskip

As the heat kernel is the building block of Brownian motion, this first result will be used as a lever for another problem: the investigation of the small-time behavior of the Brownian motion. In this case the framework has infinite dimension by its very nature, since we deal no longer with measures on $\X$ but on the path space $C([0,1],\X)$. The second goal of the paper is thus to obtain a Large Deviations Principle for the laws of a family of Brownian motions whose quadratic variation is vanishing in a sense that we shall soon make precise. In view of the previous discussion, the rate function appearing below is not surprising.

\begin{Theorem}[Theorem \ref{thm:brownian-upper}]\label{thm:main2}
Let $(\X,\sfd,\mm)$ be a proper $\RCD(K,\infty)$ space with $K \in \R$, $x \in \X$ and $\B \in \prob{C([0,1],\X)}$ be the law of the Brownian motion starting at $x$. Define, for $t \in (0,1]$ and $s \in [0,1]$,
\[
\B_t := ({\rm restr}_0^t)_*\B \qquad \textrm{with} \quad ({\rm restr}_0^t)(\gamma)(s) := \gamma_{st},\, \gamma \in C([0,1],\X)\,.
\]
Then the family $(\B_t)$ satisfies a Large Deviations Principle with speed $t^{-1}$ and rate function $I : C([0,1],\X) \to [0,\infty]$ given by
\[
I(\gamma) := \left\{\begin{array}{ll}
\displaystyle{\frac{1}{4}\int_0^1 |\dot{\gamma}_t|^2\,\d t} & \qquad \text{ if } \gamma \in \AC^2([0,1],\X),\,\gamma_0 = x, \\
+\infty & \qquad \text{ otherwise.}
\end{array}\right.
\]
\end{Theorem}
We shall derive Theorem \ref{thm:main2} from Theorem \ref{thm:main1} by relying on another very general equivalent statement of the LDP provided in \cite{Mariani18}: the family of measures $(\mu_t)$ satisfies a LDP with speed $t^{-1}$ and rate function $I$ if and only if the rescaled relative entropies $t\,H(\cdot\,|\,\mu_t)$ $\Gamma$-converge to the functional $\nu\mapsto\int I\,\d\nu$, where here $\Gamma$-convergence is intended w.r.t.\ the weak topology of $\prob{C(0,1],\X)}$. Although relying on this perspective is not strictly necessary, we believe that this approach clarifies some steps of the proofs.

In the last section of the paper we show how our results can be used in conjunction with those in \cite{{Leonard12}} to prove a suitable $\Gamma$-convergence result of the Schr\"odinger problem to the Optimal Transport problem.

\bigskip

As it is perhaps clear from this discussion, once we got Theorem \ref{thm:1} the rest will follow by  quite standard machinery related to the theory of Large Deviations, and in particular has little to do with lower Ricci curvature bounds. Thus let us conclude this introduction by focusing on the main ideas that lead to Theorem \ref{thm:1}. Introducing the Cole--Hopf and Hopf--Lax semigroups
\[
\begin{split}
Q^\eps_t(\varphi)(x) & := \eps\log\big(\h_{\eps t}e^{\varphi/\eps}\big)(x)=\eps\log\Big(\int e^{\varphi/\eps}\,\d\mu_{\eps t}[x]\Big)\,,\\
Q_t(\varphi)(x) & := \inf_{y\in\X}\Big\{\frac{\sfd^2(x,y)}{4t} + \varphi(y)\Big\}\,,
\end{split}
\]
the limiting equation \eqref{eq:visc} can be equivalently formulated as 
\[
Q^\eps_t(\varphi) \to -Q_t(-\varphi) \qquad \text{pointwise as $\eps\downarrow0$ for t=1}
\]
(the case of general $t>0$ follows by simple scaling). From a PDE perspective this amounts to proving that solutions of
\begin{equation}
\label{eq:hjbi}
\partial_t\varphi^\eps_t=|\nabla\varphi^\eps_t|^2+\eps\Delta\varphi^\eps_t
\end{equation}
converge to solutions of
\begin{equation}
\label{eq:hji}
\partial_t\varphi_t=|\nabla\varphi_t|^2\,.
\end{equation}
Note that, in line with the considerations made initially,  equation \eqref{eq:hjbi} is tightly linked --via the Laplacian-- to analysis on the base space, whereas \eqref{eq:hji} to geometry (as it is clear from the representation formula of its solutions given by the Hopf--Lax formula recalled above). We remark that the viscous-metric aspects of the Hopf--Lax formula in connection with Hamilton--Jacobi equations has been carefully studied in the independent papers \cite{AF14} and \cite{GS15}.

The validity of this sort of viscous approximation of the Hamilton--Jacobi equation in the setting of $\RCD$ spaces has been first pointed out in \cite{GigTam21}, where the setting was that of finite-dimensional $\RCD(K,N)$ spaces. Here, by following the same strategy, we push these results to the very general setting of $\RCD(K,\infty)$,  in particular not relying on any sort of local compactness.

Proving such convergence requires, among other things, to establish suitable estimates for solutions of \eqref{eq:hjbi} which are uniform in $\eps$.  Section \ref{sec:contraction} is devoted to the derivation of such estimates, which can be stated as:
\begin{equation}
\label{eq:estintro}
\begin{split}
 \Lip(\varphi_t^\eps) &\leq e^{-K\eps t}\Lip(\varphi)\,,\\
(\Delta\varphi_t^\eps)^-& \leq \|(\Delta\varphi)^-\|_{L^\infty(\X)} + 2K^-te^{2K^-\eps t}\Lip(\varphi)^2\,,
\end{split}
\end{equation}
for every $t\geq 0$. See Corollary \ref{cor:easy} for the rigorous statement (paying attention to the factor $\frac12$ used in formula \eqref{eq:hjbsemigroup}). 

In the smooth literature, such inequalities are known to be useful tools in the study of non-linear PDEs, see e.g.\ \cite[Proposition 2.1, Corollary 9.1]{lions1982generalized} where uniform bounds are established for gradients and Laplacians of solutions to Hamilton--Jacobi equations on Euclidean domains, with a vanishing viscosity term and a general convex Hamiltonian. Still in the smooth setting, the fact that lower Ricci bounds play a role in  \eqref{eq:estintro} is known at least since  \cite[Lemma 3.2]{AmStTre}.

To the best of our knowledge, little is known about these bounds in the non-smooth setting of $\RCD$ spaces and the only relevant reference is the recent \cite{MondinoSemola21}, where (an estimate close in spirit to) the Laplacian bound in \eqref{eq:estintro} has been derived in the limit case $\eps=0$ and in the setting of finite-dimensional and non-collapsed $\RCD$ spaces.



\section{Preliminaries}\label{sec:preliminaries}

\subsection{Sobolev calculus on \texorpdfstring{$\RCD$}{RCD} spaces}

By $C([0,1],\X)$ we denote the space of continuous curves with values in the metric space $(\X,\sfd)$ and for $t \in [0,1]$ the evaluation map $\e_t : C([0,1],\X) \to \X$ is defined as $\e_t(\gamma) := \gamma_t$. For the notion of absolutely continuous curve in a metric space and of metric speed see for instance Section 1.1 in \cite{AmbrosioGigliSavare08}. The collection of absolutely continuous curves on $[0,1]$ is denoted by $\AC([0,1],\X)$. By $\prob\X$ we denote the space of Borel probability measures on $(\X,\sfd)$ and by $\probt\X \subset \prob\X$ the subclass of those with finite second moment.

\medskip

Let $(\X,\sfd,\mm)$ be a complete and separable metric measure space endowed with a Borel non-negative measure which is finite on bounded sets.  We shall also assume that 
\begin{center}
    the support of $\mm$ coincides with the whole $\X$.
\end{center}
For the definition of {\bf test plans}, the {\bf Sobolev class} $S^2(\X)$ and of {\bf minimal weak upper gradient} $|D f|$ see \cite{AmbrosioGigliSavare11} (and the previous works \cite{Cheeger00, Shanmugalingam00} for alternative --but equivalent-- definitions of Sobolev functions). The Sobolev space $W^{1,2}(\X)$ is defined as $L^2(\X)\cap S^2(\X)$. When endowed with the norm $\|f\|_{W^{1,2}}^2:=\|f\|_{L^2}^2+\||Df|\|_{L^2}^2$, the space $W^{1,2}(\X)$ is a Banach space. 

The {\bf Cheeger energy} is the convex and lower-semicontinuous functional $E:L^2(\X)\to[0,\infty]$ given by
\[
E(f):=\left\{\begin{array}{ll}
\displaystyle{\frac12\int|D f|^2\,\d\mm}&\qquad \text{for }f\in W^{1,2}(\X)\,,\\
+\infty&\qquad\text{otherwise}.
\end{array}\right.
\]
$(\X,\sfd,\mm)$ is {\bf infinitesimally Hilbertian} (see \cite{Gigli12}) if $W^{1,2}(\X)$ is Hilbert. In this case the {\bf cotangent module} $L^2(T^*\X)$ (see \cite{Gigli14}) and its dual, the {\bf tangent module} $L^2(T\X)$, are canonically isomorphic, the {\bf differential} is a well-defined linear map $\d$ from $S^2(\X)$ with values in $L^2(T^*\X)$ and the isomorphism sends the differential $\d f$ to the gradient $\nabla f$. In this framework, the {\bf divergence} of a vector field is defined as (minus) the adjoint of the differential, i.e.\ we say that $v \in L^2(T\X)$ has a divergence in $L^2(\X)$, and write $v \in D({\rm div})$, provided there is a function $g \in L^2(\X)$ such that
\[
\int fg\,\d\mm = -\int \d f(v)\,\d\mm \qquad \forall f \in W^{1,2}(\X)\,.
\]
In this case $g$ is unique and is denoted ${\rm div}(v)$.

Furthermore $E$ is a Dirichlet form admitting a \emph{carr\'e du champ} given by $\langle\nabla f,\nabla g\rangle$, where $\langle\cdot,\cdot\rangle$ is the pointwise scalar product on the Hilbert module $L^2(T\X)$. The infinitesimal generator $\Delta$ of $E$, which is a linear closed self-adjoint operator on $L^2(\X)$, is called {\bf Laplacian} on $(\X,\sfd,\mm)$ and its domain denoted by $D(\Delta)\subset W^{1,2}(\X)$. A function $f \in W^{1,2}(\X)$ belongs to $D(\Delta)$ and $g = \Delta f$ if and only if
\[
\int\phi g\,\d\mm = -\int\langle\nabla\phi,\nabla f\rangle\,\d\mm, \qquad \forall \phi \in W^{1,2}(\X)\,.
\]
The flow $(\h_t)$ associated to $E$ is called {\bf heat flow} (see \cite{AmbrosioGigliSavare11}), for such a flow it holds
\[
u \in L^2(\X) \qquad \Rightarrow \qquad (\h_t u) \in C([0,\infty),L^2(\X)) \cap \AC_{\loc}((0,\infty),W^{1,2}(\X))
\]
and the curve $t\mapsto\h_tu\in L^2(\X)$ is the only solution of
\begin{equation}
\label{eq:calorel2}
\ddt\h_tu = \Delta\h_tu, \qquad \h_tu \to u\text{ as }t \downarrow 0\,,
\end{equation}
where it is intended that $(\h_tu)$ is in $\AC_{\loc}((0,\infty),L^2(\X))$ and that the derivative is intended in $L^2(\X)$. We shall denote by $\Lip(\X)$ (resp.\ $\Lip_\bs(\X)$, resp.\ $\Lip_{\bs,+}(\X)$) the space of real-valued Lipschitz (resp.\  Lipschitz  with bounded support, resp.\ Lipschitz  with bounded support and non-negative) functions on $\X$. On infinitesimally Hilbertian spaces, $\Lip_{\bs,+}(\X)$ is $W^{1,2}$-dense in the space of non-negative functions in $W^{1,2}(\X)$ and then clearly $\Lip_{\bs}(\X)$ is dense in $W^{1,2}(\X)$.

If moreover $(\X,\sfd,\mm)$ is an $\RCD(K,\infty)$ space (see \cite{AmbrosioGigliSavare11-2}), then there exists the {\bf heat kernel}, namely a function 
\[
(0,\infty)\times \X^2 \ni (t,x,y)\quad\mapsto\quad \hp_t[x](y)=\hp_t[y](x)\in (0,\infty)
\]
such that
\begin{equation}
\label{eq:rapprform}
\h_tu(x) = \int u(y)\hp_t[x](y)\,\d\mm(y)\qquad\forall t>0
\end{equation}
for every $u\in L^2(\X)$. For every $x\in \X$ and $t>0$, $\hp_t[x]$ is a probability density and thus \eqref{eq:rapprform} can be used to extend the heat flow to $L^1(\X)$ and shows that the flow is {\bf mass preserving} and satisfies the {\bf maximum principle}, i.e.
\[
f\leq c\quad\mm\ae \qquad\qquad\Rightarrow \qquad\qquad\h_tf\leq c\quad\mm\ae,\ \forall t>0\,.
\]
An important property of the heat flow on $\RCD(K,\infty)$ spaces is the {\bf Bakry--\'Emery contraction estimate} (see \cite{AmbrosioGigliSavare11-2}):
\begin{equation}
\label{eq:be}
|\d\h_tf|^2\leq e^{-2Kt}\h_t(|\d f|^2)\qquad\forall f\in W^{1,2}(\X),\ t\geq 0\,.
\end{equation}
Furthermore (see \cite{AmbrosioGigliSavare11-2}) if $u \in L^\infty(\X)$, then $\h_\cdot u(\cdot)$ belongs to $C_b((0,\infty) \times\X)$, $\h_t u$ is Lipschitz on $\X$ for all $t>0$ with Lipschitz constant controlled by
\[
\sqrt{2\int_0^t e^{2Ks}\d s}\Lip(\h_t u) \leq \|u\|_{L^\infty(\X)}
\]
and it also holds
\begin{equation}\label{eq:C0reg}
u \in C_b(\X) \qquad \Rightarrow \qquad (t,x) \mapsto \h_t u(x) \in C_b([0,\infty) \times \X)\,.
\end{equation}
We also recall that $\RCD(K,\infty)$ spaces have the {\bf Sobolev-to-Lipschitz} property (\cite{AmbrosioGigliSavare11-2, Gigli13}), i.e.
\[
f\in W^{1,2}(\X),\ |\d f|\in L^\infty(\X)\qquad\Rightarrow\qquad \exists \tilde f=f\ \mm\textrm{-a.e. with }\Lip(\tilde f)\leq\||\d f|\|_{L^\infty(\X)}\,,
\]
and thus we shall typically identify Sobolev functions with bounded differentials with their Lipschitz representative.

\bigskip

On $\RCD(K,\infty)$ spaces, the vector space of `test functions' (see \cite{Savare13})
\[
\testi\X := \Big\{ f \in D(\Delta) \cap L^{\infty}({\X}) \ :\ |\nabla f| \in L^{\infty}(\X),\ \Delta f \in L^{\infty}\cap W^{1,2}(\X) \Big\}
\]
is an algebra dense in $W^{1,2}(\X)$ for which it holds
\[
\text{$u \in \testi\X$   and $\Phi\in C^\infty(\R)$ with $\Phi(0) = 0$}\qquad\Rightarrow\qquad\text{$\Phi\circ u \in \testi\X$}\,. 
\]
An explicit method to build test functions consists in a combination of truncation and mollification via heat flow, see \cite[Proposition 6.1.8]{GigPas20}. The fact that $\testi\X$ is an algebra is based on the property
\[
\begin{split}
f\in\testi\X \qquad \Rightarrow \qquad &|\d f|^2\in W^{1,2}(\X)\quad\text{ with }\\
&\int|\d(|\d f|^2)|^2\,\d\mm\leq \||\d f|\|^2_{L^\infty}\Big(\||\d f|\|_{L^2}\||\d \Delta f|\|_{L^2}+|K|\||\d f|\|_{L^2}^2\Big)\,,
\end{split}
\]
which allows to show that the {\bf Bochner inequality} (a stronger formulation can be found in \cite{Gigli14} - see also the previous contributions \cite{Savare13, Sturm14}) holds for all $f \in \testi\X$ in the following form:
\begin{equation}
\label{eq:bochhess}
-\int \Big(\frac{1}{2}\langle\nabla h,\nabla |\nabla f|^2\rangle + h\langle\nabla f,\nabla\Delta f\rangle\Big)\d\mm \geq K\int h|\d f|^2\d\mm \qquad \forall h \in L^1 \cap W^{1,2}(\X)\,.
\end{equation}

\subsection{Optimal transport and Regular Lagrangian Flows on \texorpdfstring{$\RCD$}{RCD} spaces}

We first recall the following result on the existence of $W_2$-geodesics in $\RCD$ spaces (see \cite{Gigli12a} and \cite{RajalaSturm12}).

\begin{Theorem}\label{thm:bm}
Let $(\X,\sfd,\mm)$ be an $\RCD(K,\infty)$ space and $\mu_0,\mu_1 \in \probt\X$ with bounded support and such that $\mu_0,\mu_1 \leq C\mm$ for some $C>0$. 

Then there exists a unique geodesic $(\mu_t)$ from $\mu_0$ to $\mu_1$, it satisfies
\begin{equation}
\label{eq:linftyrcd}
\mu_t\leq C'\mm\qquad\forall t\in[0,1]\text{ for some }C'>0\,,
\end{equation}
and there is a unique \emph{lifting} $\ppi$ of it, i.e.\ a unique measure $\ppi\in\prob{C([0,1],X)}$ such that $(\e_t)_*\ppi=\mu_t$ for every $t\in[0,1]$ and 
\[
\iint_0^1|\dot\gamma_t|^2\,\d t\,\d\ppi(\gamma)=W_2^2(\mu_0,\mu_1)\,.
\]
\end{Theorem}

Given $f : \X \to \R$ the {\bf local Lipschitz constant} $\lip(f) : \X \to [0,\infty]$ is defined as 0 on isolated points and otherwise as
\[
\lip f(x) := \limsup_{y \to x}\frac{|f(x) - f(y)|}{\sfd(x,y)}\,.
\]
If $f$ is Lipschitz, then its {\bf Lipschitz constant} is denoted by $\Lip f$. With this notion we can recall some properties of the {\bf Hopf--Lax semigroup} in metric spaces, also in connection with optimal transport. For $f : \X \to \R \cup \{+\infty\}$ and $t > 0$ we define the function $Q_tf : \X \to \R \cup \{-\infty\}$ as
\begin{equation}
\label{eq:hli}
Q_tf(x) := \inf_{y\in \X}\Big\{f(y) + \frac{\sfd^2(x,y)}{2t}\Big\}
\end{equation}
and set 
\[
t_* := \sup\{t > 0 \,:\, Q_tf(x) > -\infty \textrm{ for some } x \in \X\}\,.
\]
It is worth saying that $t_*$ does not actually depend on $x$, since if $Q_tf(x) > -\infty$, then $Q_sf(y) > -\infty$ for all $s \in (0,t)$ and all $y \in \X$. With this premise we have the following result (see \cite{AmbrosioGigliSavare11}):

\begin{Proposition}\label{pro:11}
Let $(\X,\sfd)$ be a length space and $f : \X \to \R \cup \{+\infty\}$. Then for all $x \in \X$ the map $(0,t_*) \ni t \mapsto Q_tf(x)$ is locally Lipschitz and
\[
\ddt Q_tf(x) + \frac{1}{2}\Big(\lip Q_tf(x)\Big)^2 = 0\qquad {\rm a.e.}\ t \in (0,t_*)\,.
\]
\end{Proposition}
%

With regard to the notion of Regular Lagrangian Flow, introduced by L.\ Ambrosio and the last-named author in \cite{Ambrosio-Trevisan14} as the generalization to $\RCD$ spaces of the analogous concept existing on $\R^d$ as proposed by Ambrosio in \cite{Ambrosio04}, the definition is the following:

\begin{Definition}[Regular Lagrangian Flow]
Given $(v_t)\in L^1([0,1],L^2(T\X))$, the function $F:[0,1]\times \X\to \X$ is a Regular Lagrangian Flow for $(v_t)$ provided:
\begin{itemize}
\item[(i)] $[0,1]\ni t \mapsto F_t(x)$ is continuous for every $x\in \X$;
\item[(ii)] for every $f\in\testi \X$ and $\mm$-a.e.\ $x$ the map $t\mapsto f(F_t(x))$ belongs to $W^{1,1}([0,1])$ and
\[
\ddt  f(F_t(x))=\d f(v_t)(F_t(x))\qquad {\rm a.e.}\ t\in[0,1]\,;
\]
\item[(iii)] it holds
\[
(F_t)_*\mm\leq C\mm\qquad\forall t\in[0,1]
\]
for some constant $C>0$ called compression constant.
\end{itemize}
\end{Definition}

In \cite[Theorems 4.3 and 4.6]{Ambrosio-Trevisan14} the authors prove that under suitable assumptions on the $v_t$'s, Regular Lagrangian Flows exist and are unique. We shall use the following formulation of their result (weaker than the one provided in \cite{Ambrosio-Trevisan14}). Since in this manuscript we are going to deal only with gradient vector fields, the space $W^{1,2}_C(T\X)$ appearing in the statement below needs not be defined.

\begin{Theorem}\label{thm:RLF}
Let $(\X,\sfd,\mm)$ be an $\RCD(K,\infty)$ space and $(v_t)\in L^1([0,1],W^{1,2}_C(T\X))$ be such that $v_t\in D({\rm div})$ for a.e.\ $t$ and
\[
{\rm div}(v_t) \in L^1([0,1],L^2(\X)) \qquad ({\rm div}(v_t))^- \in L^1([0,1],L^\infty(\X))\,.
\]
Then there exists a unique, up to $\mm$-a.e.\ equality, Regular Lagrangian Flow $F$ for $(v_t)$.

For such flow, the quantitative bound
\begin{equation}
\label{eq:quantm}
(F_t)_*\mm\leq \exp\Big(\int_0^1\|({\rm div}(v_t))^-\|_{L^\infty(\X)}\,\d t\Big)\mm
\end{equation}
holds for every $t\in[0,1]$ and for $\mm$-a.e.\ $x$ the curve $t\mapsto F_t(x)$ is absolutely continuous and its metric speed ${\rm ms}_t({F_{\cdot}}(x))$ at time $t$ satisfies
\begin{equation}
\label{eq:quants}
{\rm ms}_t({F_{\cdot}}(x))=|v_t|(F_t(x))\qquad {\rm a.e.}\ t \in [0,1]\,.
\end{equation}
\end{Theorem}

To be precise, \eqref{eq:quants} is not explicitly stated in \cite{Ambrosio-Trevisan14}; its proof is anyway not hard and can be obtained, for instance, following the arguments in \cite{Gigli14}.

As anticipated, for the purposes of this manuscript the general definition of the space $W^{1,2}_C(T\X)$ is not relevant: it is sufficient to know that
\begin{equation}
\label{eq:w12c}
\|\nabla\varphi\|_{W^{1,2}_C}\leq C(\|\Delta\varphi\|_{L^2}+\|\nabla\varphi\|_{L^2}),\qquad\forall\varphi\in\testi\X\,,
\end{equation}
see \cite[Corollary 3.3.9]{Gigli14} for the proof.

Finally, we shall also need the following compactness criterion for Regular Lagrangian Flows, stated at the level of the induced test plans. The formulation we adopt is tailored to our framework and is weaker than the original one, contained in \cite{AmbrosioStraTrevisan17} (see Proposition 7.8 therein).

\begin{Theorem}\label{thm:RLFstability}
Let $(\X,\sfd,\mm)$ be an $\RCD(K,\infty)$ space, $(\mu^n) \subset \prob\X$ with $\mu^n \ll \mm$, $(F^n)$ a family of Regular Lagrangian Flows relative to $(v^n_t)$ with compression constants $C_n$ and set $\ppi^n := (F^n_{.})_*\mu^n$, where $F^n_{.} : \X \to C([0,1],\X)$ is the $\mm$-a.e.\ defined map sending $x$ to $t \mapsto F_t^n(x)$. Assume that
\begin{equation}\label{eq:stability1}
\lim_{R \to \infty}\sup_n (\e_0)_*\ppi^n(\X \setminus B_R(\bar{x})) = 0
\end{equation}
for some $\bar{x} \in \X$,
\begin{equation}\label{eq:stability2}
\sup_n C_n < \infty \qquad\textrm{and}\qquad \sup_n \iint_0^1 |v^n_t|^2\d t\d\mm < \infty\,.
\end{equation}
Then $(\ppi^n) \subset \prob{C([0,1],\X)}$ is tight and every limit point $\ppi$ is concentrated on $\AC([0,1],\X)$.
\end{Theorem}

\subsection{Large deviations and \texorpdfstring{$\Gamma$}{Gamma}-convergence}

The notion of Large Deviations is meaningful in a very abstract and general framework, namely in a Polish space $(\X,\tau)$. In what follows:
\begin{itemize}
\item[-] $I:\X\to[0,\infty]$ is  a given function, called \emph{rate} function, with \textbf{compact} sublevels (and in particular lower semicontinuous). 
\item[-]   $(\mu_t) \subset \prob\X$, $t\in(0,1)$ is a given family of Borel probability measures. 
\end{itemize}

\begin{Definition}\label{def:ldp}
We say that $(\mu_t)$ satisfies:
\begin{itemize}
\item[(i)] a Large Deviations lower bound with speed $t^{-1}$ and rate $I$ if for each open set $O \subset \X$
\[
\liminf_{t\downarrow0} t\log\mu_t(O) \geq -\inf_{x \in O}I(x)\,;
\]
\item[(ii)] a Large Deviations upper bound with speed $t^{-1}$ and rate $I$ if for each closed set $C \subset \X$
\[
\limsup_{t\downarrow0} t\log\mu_t(C) \leq -\inf_{x \in C}I(x)\,.
\]
\end{itemize}
If both lower and upper bounds hold with same rate and speed, we say that $(\mu_t)$ satisfies a  Large Deviations Principle.
\end{Definition}

Thus recalling the Portmanteau Theorem, Large Deviations can be thought of as an `exponential' analog of weak convergence of measures. 

As discussed in detail in  \cite{Mariani18}, the concept of Large Deviation is also tightly linked to that of $\Gamma$-convergence of (rescaled) entropy functionals. Let us recall that  given $\nu \in \prob\X$, the entropy $H(\cdot\,|\,\nu):\prob\X\to[0,\infty]$ relative to $\nu$  is defined  as
\[
H(\sigma\,|\,\nu):=\left\{\begin{array}{ll}
\displaystyle{\int\rho\log(\rho)\,\d{\nu}} & \qquad \text{ if } \sigma = \rho\nu,\\
+\infty & \qquad \text{ if } \sigma \not\ll {\nu}.
\end{array}\right.
\]
Note that Jensen's inequality and the fact that $\nu\in \prob\X$ grant that $\int\rho\log(\rho)\,\d\nu$ is well defined and non-negative. We also recall De Giorgi's concept of $\Gamma$-convergence (see e.g.\ \cite{DM12}), that we shall only state for family of functionals on $\prob\X$ and for weak convergence of measures:

\begin{Definition}[$\Gamma$ convergence]
Let $(F_t)$, $t\geq 0$, be given real-valued functionals on $\prob\X$. We say that:
\begin{itemize}
\item[(i)] $\Gamma$-$\limsup_t F_t\leq F_0$ provided for every $\nu_0\in\prob\X$ there is $\nu_t\rightharpoonup\nu_0$ such that
\[
\limsup_{t\downarrow0}F_t(\nu_t)\leq F_0(\nu_0)\,.
\] 
\item[(ii)] $\Gamma$-$\liminf_t F_t\geq F_0$ provided for every $(\nu_t)\subset\prob\X$  such that $\nu_t\rightharpoonup\nu_0$ we have
\[
\liminf_{t\downarrow0}F_t(\nu_t)\geq F_0(\nu_0)\,.
\] 
\end{itemize}
If both $(i)$ and $(ii)$ hold we say that the family $(F_t)$ $\Gamma$-converges to $F_0$.
\end{Definition}

We then recall the following general equivalences. Let us emphasize the presence of the `minus' sign in Definition \ref{def:ldp} and the fact that for given $\nu\in\prob\X$ and $E\subset\X$ Borel we have $H(\nu(E)^{-1}\nu\restr E\,|\,\nu)=-\log(\nu(E))$: this helps understanding why Large Deviations lower bounds are related to the $\Gamma$-$\limsup$ inequality, and and analogously Large Deviations upper bounds to $\Gamma$-$\liminf$.

\begin{Theorem}\label{thm:lower-ldp}
The following are equivalent:
\begin{itemize}
\item[(i)] $(\mu_t)$ satisfies a Large Deviations lower bound with speed $t^{-1}$ and rate $I$;
\item[(ii)] for every $\varphi\in C_b(\X)$ it holds
\[
\liminf_{t\downarrow0 }t\log\int e^{\frac\varphi t}\,\d\mu_t \geq \sup_{x \in \X}\{\varphi(x) - I(x)\}\,;
\]
\item[(iii)] for every $\nu \in \prob\X$ it holds
\begin{equation}
\label{eq:glimsld}
\Big(\Gamma\textrm{-}\limsup_{t\downarrow0}\, tH(\cdot\,|\,\mu_t)\Big)(\nu) \leq \int I\,\d\nu\,;
\end{equation}
\item[(iii')] for every $x \in \X$ the inequality \eqref{eq:glimsld}  holds for $\nu:=\delta_x$.
\end{itemize}
\end{Theorem}

\begin{Theorem}\label{thm:upper-ldp} 
The following are equivalent:
\begin{itemize}
\item[(i)] $(\mu_t)$ satisfies a Large Deviations  upper bound with speed $t^{-1}$ and rate $I$;
\item[(ii)] for every $\varphi\in C_b(\X)$ it holds
\[
\limsup_{t\downarrow0 }t\log\int e^{\frac\varphi t}\,\d\mu_t \leq \sup_{x \in \X}\{\varphi(x) - I(x)\}\,;
\]
\item[(iii)] for every $\nu \in \prob\X$ it holds
\begin{equation}
\label{eq:glimild}
\Big(\Gamma\textrm{-}\liminf_{t\downarrow0}\, tH(\cdot\,|\,\mu_t)\Big)(\nu) \geq \int I\,\d\nu\,;
\end{equation}
\item[(iii')] for every $x \in \X$ the inequality \eqref{eq:glimild}  holds for $\nu:=\delta_x$.
\end{itemize}
\end{Theorem}
The equivalence of $(i),(ii)$ in these statements is standard in the field, see for instance \cite[Lemma 4.3.4, 4.3.6, 4.4.5, 4.4.6]{DZ98}  for the proof. The fact that these are equivalent to $(iii),(iii')$ is proved in \cite[Theorem 3.4, 3.5]{Mariani18}.

\section{Contraction estimates for gradient and Laplacian}\label{sec:contraction}

Aim of this section is to get gradient and Laplacian estimates for $\log\h_t e^\varphi$ in terms of the quantities $\| |\nabla\varphi|\|_{L^\infty(\X)}$ and $\|(\Delta \varphi)^-\|_{L^\infty(\X)}$.

The proofs are based on a comparison principle, valid in general infinitesimally Hilbertian spaces $(\Y,\sfd_\Y,\mm_\Y)$. To formulate it, we shall denote   by $W^{-1,2}(\Y)$ the dual of $W^{1,2}(\Y)$. As usual, the fact that $W^{1,2}(\Y)$ embeds in $L^2(\Y)$ with dense image allows to see $L^2(\Y)$ as a dense subset of $W^{-1,2}(\Y)$, where $f\in L^2(\Y)$ is identified with the mapping $W^{1,2}(\Y)\ni g\mapsto \int fg\,\d\mm_\Y$. 

\begin{Lemma}\label{le:kyoto}
Let $(\Y,\sfd_\Y,\mm_\Y)$ be an infinitesimally Hilbertian space and $T>0$. Then
\begin{equation}
\label{eq:inclspazi}
L^\infty([0,T],W^{1,2}(\Y))\ \cap\  C([0,T],W^{-1,2}(\Y))\quad \subset\quad  C([0,T],L^2(\Y))\,.
\end{equation}
Moreover, for $(g_t)\in L^\infty([0,T],W^{1,2}(\Y))\cap \AC([0,T],W^{-1,2}(\Y))$ the function $[0,T]\ni t \mapsto \frac12\int|(g_t)^+|^2\,\d\mm_\Y$ is  absolutely continuous  and it holds
\begin{equation}
\label{eq:tesilemma}
\ddt \frac12\int|(g_t)^+|^2\,\d\mm_\Y = \int (g_t)^+\frac{\d}{\d t}g_t\,\d\mm_\Y\,,\qquad {\rm a.e.}\ t,
\end{equation}
where the integral in the right-hand side is intended as the coupling of $(g_t)^+\in W^{1,2}(\Y)$ with $\frac\d{\d t}g_t\in W^{-1,2}(\Y)$.
\end{Lemma}

\begin{proof} It is readily verified that both the $W^{1,2}$-norm and $L^2$-norm are lower semicontinuous w.r.t.\ $W^{-1,2}$-convergence. Hence for $(g_t)$ in the space in the left-hand side in \eqref{eq:inclspazi} we have $\sup_{t\in[0,T]}\|g_t\|_{W^{1,2}}<\infty$. Now let $t_n \to t$ and note that, by what just said, up to passing to a non-relabeled subsequence we can assume that $g_{t_n}\weakto g$ in $W^{1,2}(\Y)$ for some $g\in W^{1,2}(\Y)$. Since in addition we know that $g_{t_n}\to g_t $ in $W^{-1,2}(\Y)$ we must have $g=g_t$ and by the coupling of strong and weak convergences we see that
\[
\int g_{t_n}^2\,\d\mm_\Y=\la g_{t_n},g_{t_n}\ra_{W^{1,2},W^{-1,2}}\to \la g_{t},h_{t}\ra_{W^{1,2},W^{-1,2}}=\int g_t^2\,\d\mm\,.
\]
Since $W^{1,2}$-weak convergence implies $L^2$-weak convergence, we have $L^2$-weak convergence and convergence of $L^2$-norms, so \eqref{eq:inclspazi} is proved.

Now let $(g_t) \in \AC([0,T],L^2(\Y))$, then $\Phi(t) := \frac12\int|(g_t)^+|^2\,\d\mm_\Y$ is locally absolutely continuous on $(0,\infty)$ and, setting $\phi(z) := z^+ = \max\{0,z\}$, \eqref{eq:tesilemma} follows by observing that
\[
\Phi'(t) = \int\phi(g_t)\ddt\phi(g_t)\,\d\mm_\Y = \int\phi'(g_t)\,\phi(g_t)\ddt g_t\,\d\mm_\Y = \int\phi(g_t)\ddt g_t \,\d\mm_\Y\,.
\]
For the general case we argue by approximation via the heat flow. Fix $\eps>0$ and note that the fact that $\h_\eps$ is a contraction in $W^{1,2}(\Y)$ and a bounded operator from $L^2(\Y)$ to $W^{1,2}(\Y)$ yield the inequalities
\[
\begin{split}
& \|\h_\eps f\|_{L^2} = \sup_{\|g\|_{L^2} \leq 1}\int \h_\eps f\,g\,\d\mm_\Y \leq \sup_{\|g\|_{L^2} \leq 1}\|\h_\eps g\|_{W^{1,2}}\|f\|_{W^{-1,2}} \leq C_\eps\|f\|_{W^{-1,2}}\,, \\
& \|\h_\eps f\|_{W^{-1,2}} = \sup_{\|g\|_{W^{1,2}} \leq 1}\int \h_\eps f\,g\,\d\mm_\Y \leq \sup_{\|g\|_{W^{1,2}} \leq 1}\|\h_\eps g\|_{W^{1,2}}\|f\|_{W^{-1,2}} \leq \|f\|_{W^{-1,2}}\,,
\end{split}
\]
for all $f \in L^2(\Y)$, which together with the density of $L^2(\Y)$ in $W^{-1,2}(\Y)$ ensures that $\h_\eps$ can be uniquely extended to a bounded linear operator from $W^{-1,2}(\Y)$ to $L^2(\Y)$ which is also a contraction when seen with values in $W^{-1,2}(\Y)$. It is then clear that $\h_\eps f \to f$ in $W^{-1,2}(\Y)$ as $\eps \downarrow 0$ for every $f \in W^{-1,2}(\Y)$. It follows that for $(g_t)$ as in the assumption we have $(\h_\eps g_t) \in \AC([0,T],L^2(\Y))$, so that, by what previously said, the conclusion holds for such curve and writing the identity \eqref{eq:tesilemma} in integral form we have
\[
\frac12\int | (\h_\eps g_{t_1})^+|^2- | (\h_\eps g_{t_0})^+|^2\,\d\mm_\Y = \int_{t_0}^{t_1}\int  \big(\h_{\eps}g_t\big)^+\h_\eps\Big(\ddt g_t\Big)\,\d\mm_\Y\,\d t \qquad\forall 0 \leq t_0 \leq t_1 \leq T\,.
\] 
Letting $\eps \downarrow 0$, using the continuity at $\eps=0$ of $\h_\eps$ seen as an operator on each of the spaces $W^{1,2}(\Y),L^2(\Y),W^{-1,2}(\Y)$, and the continuity of $g \mapsto g^+$ as map from $W^{1,2}(\Y)$ with the strong topology to $W^{1,2}(\Y)$ with the weak one (which follows from the continuity of the same operator in $L^2(\Y)$ together with the fact that it decreases the $W^{1,2}$-norm), we obtain
\begin{equation}
\label{eq:tesik}
\frac12\int |g_{t_1}^+|^2 - |g_{t_0}^+|^2\,\d\mm_\Y = \int_{t_0}^{t_1}\int (g_t)^+ \ddt g_t \,\d\mm_\Y\,\d t \qquad \forall 0 \leq t_0 \leq t_1 \leq T\,.
\end{equation}
Since clearly we have
\[
\bigg|\int_{t_0}^{t_1}\int (g_t)^+ \ddt g_t \,\d\mm_\Y\,\d t\bigg| \leq \|(g_t)\|_{L^\infty([t_0,t_1],W^{1,2})}  \|(\tfrac \d{\d t} g_t)\|_{L^1([t_0,t_1],W^{-1,2})}\,,
\] 
the identity \eqref{eq:tesik} and the assumptions on $(g_t)$ grant the local absolute continuity of $t \mapsto \frac12\int |g_t^+|^2\,\d\mm_\Y$. Then the conclusion follows by differentiating \eqref{eq:tesik}.
\end{proof}
%

In the following statement we are going to prove a general comparison principle between super- and sub-solutions of the equation
\begin{equation}
\label{eq:persupsub}
\frac{\d}{\d t} u_t = \Delta u_t + a_1u_t + \langle \nabla u_t, v_t \rangle + a_2\,.
\end{equation}

\begin{Proposition}[A comparison principle]\label{pro:3}
Let $(\Y,\sfd_\Y,\mm_\Y)$ be an infinitesimally Hilbertian space,  $a_1,a_2,T \in \R$ with $T>0$ and $(v_t)$ be a Borel family of vector fields with 
\begin{equation}
\label{eq:boundvt}
(|v_t|)\in  L^\infty([0,T],L^\infty(\Y))\,.
\end{equation}
Also, let $(F_t),(G_t)$ be time-dependent families of functions belonging to the space
\begin{equation}
\label{eq:spazi2}
L^\infty([0,T],W^{1,2}(\Y))  \cap \AC([0,T],W^{-1,2}(\Y))\,.
\end{equation}
Assume that $(F_t),(G_t)$ are respectively a weak super- and weak sub-solution of \eqref{eq:persupsub}, meaning that for every $h\in \Lip_{\bs, +}(\Y)$ the functions $t\mapsto \int hF_t\,\d\mm_\Y,\int hG_t\,\d\mm_\Y$ (that are  absolutely continuous on $[0,T]$ by the assumptions on $(F_t),(G_t)$) satisfy
\begin{eqnarray}
&& \ddt \int h F_t\d\mm_\Y \geq -\int \langle\nabla h,\nabla F_t\rangle\,\d\mm_\Y + \int h\Big(a_1F_t + \langle \nabla F_t,v_t \rangle + a_2\Big)\d\mm_\Y\,, \nonumber \\
&& \ddt \int h G_t\d\mm_\Y \leq -\int \langle\nabla h,\nabla G_t\rangle\,\d\mm_\Y + \int h\Big(a_1G_t + \langle \nabla G_t,v_t \rangle + a_2\Big)\d\mm_\Y\,. \nonumber
\end{eqnarray}
for a.e.\ $t\in [0,T]$. Assume  that $F_0 \geq G_0$ $\mm_\Y$-a.e. Then $F_t \geq G_t$ $\mm_\Y$-a.e.\ for every $t \in [0,T]$.
\end{Proposition}

\begin{proof} 
Recalling that $\Lip_{\bs,+}(\Y)$ is dense in the space of non-negative functions in $W^{1,2}(\Y)$ and integrating in time the inequalities defining super/sub-solutions, we see that such inequalities are valid for every $h\in W^{1,2}(\Y)$ non-negative. Then from the fact that $W^{-1,2}(\Y)$ has the Radon--Nikod\'ym property (because it is Hilbert) we see that $(F_t)$ seen as curve with values in $W^{-1,2}(\Y)$ must be differentiable at a.e.\ $t$ (see \cite[Chapter VII.6]{DiestelUhl77} for equivalent definitions of the Radon--Nikod\'ym property and \cite[Chapter III.3]{DiestelUhl77} for its validity in the case of Hilbert spaces). It is then clear that for every point of differentiability $t$, the inequality
\begin{equation}
\label{eq:16f}
\frac{\d}{\d t}\int h F_t\d\mm_\Y \geq -\int \la\nabla h,\nabla F_t\ra \d\mm_\Y + \int h\Big(a_1F_t + \langle \nabla F_t,v_t \rangle + a_2\Big)\d\mm_\Y
\end{equation}
holds for every $h \in W^{1,2}(\Y)$ non-negative, i.e.\ that the set of $t$'s for which \eqref{eq:16f} holds is independent on $h$. The analogous property holds for $(G_t)$.

Now observe that $H_t := G_t-F_t$ satisfies the assumptions of Lemma \ref{le:kyoto}, thus the function  $\Phi(t) := \frac12\int| H_t^+|^2\,\d\mm_Y$ is  absolutely continuous on $[0,T]$ and 
\[
\Phi'(t) = \int H_t^+ \ddt H_t \,\d\mm_Y\,,
\]
where the right-hand side is intended as the coupling of $\ddt H_t  \in W^{-1,2}(\Y)$ and the function $H_t^+ \in W^{1,2}(\Y)$. Fix $t$ which is a differentiability point of both $(F_t)$ and $(G_t)$ and pick $h :=H_t^+= (G_t-F_t)^+$ in \eqref{eq:16f} and in the analogous inequality for $(G_t)$ to obtain
\[
\begin{split}
\Phi'(t) & \leq \int-\langle\nabla H_t^+,\nabla H_t\rangle + H_t^+ \big(a_1H_t + \langle\nabla H_t,v_t\rangle\big)\,\d\mm_\Y\,.
\end{split}
\]
Now fix $T>0$, let $C_T := {\rm ess\,sup}_{t\in[0,T]}\||v_t|\|_{L^\infty}<\infty$ and use the trivial identities  $\langle\nabla f^+,\nabla f\rangle = |\nabla f^+|^2$, $f^+f=|f^+|^2$, and $f^+\nabla f = f^+\nabla f^+$ valid $\mm_\Y$-a.e.\ for every $f \in W^{1,2}(\Y)$, to obtain
\[
\begin{split}
\Phi'(t) & \leq \int -|\d H_t^+|^2 + a_1|H_t^+|^2 + |H_t^+||\d H^+_t|C_T \d\mm_\Y \leq \Big(a_1 + \frac{C_T^2}{2}\Big) \Phi(t)\qquad \textrm{a.e. } t \in (0,T)\,,
\end{split}
\]
where the second inequality follows from $|H_t^+||\d H^+_t|C_T \leq \frac12 |H_t^+|^2 C_T^2 + \frac12 |\d H^+_t|^2$. By Gr\"onwall's lemma we deduce that
\[
\Phi(s)\leq e^{(2a_1+C_T^2)(s-t)}\Phi(t)\qquad\forall t,s\in(0,T),\ t<s
\]
and since $\Phi$ is continuous at $t=0$ (by \eqref{eq:inclspazi} and our assumptions on $(F_t),(G_t)$) we can let $t\downarrow0$ and deduce that $\Phi(s)\leq e^{(a_1+C_T^2)(s)}\Phi(0)$ holds for every $s\in(0,T)$. Since $\Phi(0)=0$, the conclusion follows.
\end{proof}

In the forthcoming analysis it will be useful to keep in mind the following simple result:

\begin{Lemma}\label{le:stimebase}
Let $(\X,\sfd,\mm)$ be an  $\RCD(K,\infty)$ space with $K \in \R$, $\phi\in\testi \X$. Put 
\[
\phi_t:=\log\big(\h_t(e^\phi)\big)\,.
\]
Then $\phi_t\in\testi\X$ for every $t\geq 0$ with
\begin{equation}
\label{eq:dDphi}
\d\phi_t=\frac{\d \h_t(e^\phi)}{\h_t(e^\phi)}\qquad\text{and}\qquad\Delta \phi_t=\frac{\Delta\h_te^\phi}{\h_te^\phi}-\frac{|\nabla\h_te^\phi|^2}{|\h_te^\phi|^2}\,,
\end{equation}
and the estimates
\begin{subequations}
\begin{align}
\label{eq:boundlinfty}
\sup_{t\geq 0}\|\phi_t\|_\infty&\leq \|\phi\|_\infty\,,\\
\label{eq:linftygradlap}
\sup_{t\in[0,T]} \||\nabla\phi_t|^2\|_\infty+\|\Delta\phi_t\|_\infty &<\infty\,,\\
\label{eq:w12gradlap}
\sup_{t\in[0,T]} \||\nabla\phi_t|^2\|_{W^{1,2}}+\|\Delta\phi_t\|_{W^{1,2}}&<\infty\,,
\end{align}
\end{subequations}
hold for every $T>0$. Finally, the map $t\mapsto \phi_t$ belongs to $\AC_{\loc}([0,\infty),W^{1,2}(\X))$ with  derivative given by
\begin{equation}\label{eq:hjb-eta2}
\ddt\phi_t = |\nabla\phi_t|^2 + \Delta\phi_t  \qquad \textrm{for a.e. } t>0\,.
\end{equation}
\end{Lemma}

\begin{proof} The continuity of $\phi_t$ follows from \eqref{eq:C0reg}. Formulas \eqref{eq:dDphi} are obvious by direct computation and so is
\begin{equation}
\label{eq:dddphi}
\d\Delta\phi_t=\frac{\d\Delta\h_te^\phi}{\h_te^\phi}-\frac{\Delta\h_te^\phi\,\d\h_te^\phi}{|\h_te^\phi|^2}-\frac{\d|\nabla\h_te^\phi|^2}{|\h_te^\phi|^2}+2\frac{\h_te^\phi|\nabla\h_te^\phi|^2\d\h_te^\phi}{|\h_te^\phi|^3}\,.
\end{equation}
Then all the stated estimates follow from the  weak maximum principle for the heat flow, the uniform bound $e^\phi\geq e^{\inf\phi}>0$ and 
\[
\begin{split}
\|\d \h_t(e^\phi)\|_\infty& \stackrel{\eqref{eq:be}}\leq e^{-Kt}\|\d e^\phi\|_\infty\leq e^{-Kt}e^{\sup \phi}\|\d \phi\|_\infty\,,\\
\|\Delta\h_t e^\phi\|_\infty&\leq \|\Delta e^\phi\|_\infty =\|e^\phi(\Delta \phi+|\d \phi|^2)\|_\infty\leq e^{\sup \phi}(\|\Delta \phi\|_\infty+\|\d \phi\|_\infty^2)\,.
\end{split}
\]
For the last claim we pick a finite Borel measure $\mm'\leq\mm$ with the same negligible sets of $\mm$ and we note that from \eqref{eq:calorel2} and \eqref{eq:dDphi} we have $(\phi_t)\in C([0,\infty),L^2(\X,\mm'))\cap \AC_{\loc}((0,\infty),L^2(\X,\mm'))$ with derivative given by the formula \eqref{eq:hjb-eta2}. Integrating we obtain
\[
\phi_s-\phi_t=\int_t^s |\nabla\phi_r|^2 + \Delta\phi_r\,\d r\,, \qquad\forall 0\leq t<s\,,
\]
where the integrals are in the sense of Bochner in the space $L^2(\X,\mm')$. Thus this formula is valid also $\mm'$-, and hence $\mm$-, almost everywhere, so that the estimates \eqref{eq:w12gradlap} and the fact that $\phi_0=\phi\in W^{1,2}(\X,\sfd,\mm)$ give the last claim and, in particular, the fact that $\phi_t\in L^2(\X,\mm)$ thus establishing $\phi_t\in\testi\X$.
\end{proof}
We are ready to turn to the main result of this section:
\begin{Proposition}\label{pro:kyoto}
Let $(\X,\sfd,\mm)$ be an $\RCD(K,\infty)$ space with $K \in \R$, $\phi\in\testi\X$ and put $\phi_t := \log(\h_te^\phi)$ for all $t \geq 0$. Then for every $t \geq 0$ it holds
\begin{align}
\label{eq:lipcontr}
\|\nabla\phi_t\|_\infty &\leq e^{-Kt}\|\nabla\phi_0\|_{L^\infty(\X)}\,, \\
\label{eq:lapcontr}
\|(\Delta\phi_t)^-\|_\infty &\leq \|(\Delta\phi_0)^-\|_{L^\infty(\X)}+ 2tK^-e^{2K^-t}\|\nabla\phi_0\|^2_{L^\infty(\X)}\,,
\end{align}
where $K^- := \max\{0,-K\}$.
\end{Proposition}

\begin{proof} We prove first \eqref{eq:lipcontr} and then \eqref{eq:lapcontr}.

\noindent{\underline{Proof of \eqref{eq:lipcontr}.}} Put $G_t := |\nabla\phi_t|^2$ and note that as a result of Lemma \ref{le:stimebase} above (and the trivial inclusion $\AC([0,T],L^2(\X))\subset \AC([0,T],W^{-1,2}(\X))$) we have that $(G_t)$ belongs to the space \eqref{eq:spazi2} for every $T>0$.  Moreover, from \eqref{eq:hjb-eta2} we get that the derivative of $(G_t)$, intended in the space $L^2(\X)$, is given by
\[
\ddt G_t = 2\langle\nabla\phi_t,\nabla G_t\rangle + 2\langle\nabla\phi_t, \nabla\Delta\phi_t\rangle \qquad \textrm{a.e. } t>0\,.
\]
Applying Bochner inequality \eqref{eq:bochhess} to $f:=\phi_t$ we deduce that for every $h\in \Lip_{\bs,+}(\X)$  it holds
\[
\ddt\int h G_t\,\d\mm \leq \int -\langle\nabla h,\nabla G_t\rangle + 2h\Big(\la\nabla\phi_t,\nabla G_t\ra-K G_t\Big)\,\d\mm,\qquad \textrm{a.e. } t>0\,,
\]
thus showing that $(G_t)$ is a weak subsolution of \eqref{eq:persupsub} in the sense of Proposition \ref{pro:3} with
\begin{equation}
\label{eq:data1}
a_1 = -2K \qquad a_2 = 0 \qquad v_t = 2\nabla\phi_t\,.
\end{equation}
On the other hand, it is clear that the function $F_t(x):= e^{-2Kt}\|G_0\|_{L^\infty(\X)}$ belongs to the space \eqref{eq:spazi2} for every $T>0$ and is a super-solution of \eqref{eq:persupsub} with the data \eqref{eq:data1}.

Since $F_0\geq G_0$ holds $\mm$-a.e.\ and the bound \eqref{eq:linftygradlap} ensures that the $v_t$'s defined above satisfy \eqref{eq:boundvt}, the claim \eqref{eq:lipcontr} follows from Proposition \ref{pro:3}.

\noindent{\underline{Proof of \eqref{eq:lapcontr}.}}  Fix $T>0$,  put $F_t := \Delta\phi_t$ and note that \eqref{eq:w12gradlap} ensures that $(F_t)$ belongs to $ L^\infty([0,T],W^{1,2}(\X)) $. Also, for every $g\in W^{1,2}(\X)$ with $\|g\|_{W^{1,2}}\leq 1$, putting for brevity $f:=\phi_s-\phi_t$ for fixed $0<t<s$, we have
\[
\begin{split}
\int g (F_s-F_t)\,\d\mm&=-\int\nabla  g\cdot\nabla f\,\d\mm\leq \||\nabla g|\|_{L^{2}} \|f\|_{W^{1,2}}\leq  \|f\|_{W^{1,2}}\,.
\end{split}
\]
Taking the $\sup$ in $g$ as above we conclude that $\|F_s-F_t \|_{W^{-1,2}}\leq \|\phi_s-\phi_t \|_{W^{1,2}}$, thus  recalling the last claim in Lemma \ref{le:stimebase}   we conclude that $(F_t)$ is in \eqref{eq:spazi2}. Also, for every $h \in \Lip_{\bs,+}(\X)$ we have
\[
\begin{split}
\frac{\d}{\d t}\int hF_t\,\d\mm&=-\frac{\d}{\d t}\int\nabla h\cdot\nabla\phi_t\,\d\mm\\
\text{(by \eqref{eq:hjb-eta2})}\qquad&= -\int\nabla h\cdot\nabla(|\nabla\phi_t|^2)+\nabla h\cdot\nabla F_t\,\d\mm\\
\text{(by \eqref{eq:bochhess})}\qquad&\geq \int 2h\big(\nabla\phi_t\cdot\nabla F_t+K|\nabla\phi_t|^2 \big)+\nabla h\cdot\nabla F_t\,\d\mm\\
\text{(by \eqref{eq:lipcontr})}\qquad&\geq \int 2h\big(\nabla\phi_t\cdot\nabla F_t-K^-e^{2K^-T}\||\nabla\phi_0|\|^2_{L^\infty} \big)+\nabla h\cdot\nabla F_t\,\d\mm\,,
\end{split}
\]
for a.e.\ $t\in[0,T]$.  In other words, $(F_t)$ is a weak supersolution of \eqref{eq:persupsub} on $[0,T]$ with 
\[
a_1 = 0 \qquad a_2 = -2K^-e^{2K^-T}\||\nabla\phi_0|\|^2_{L^\infty(\X)} \qquad v_t = 2\nabla\phi_t\,.
\]
On the other hand, it is easily verified that the function
\[
G_t(x) := -\|(\Delta\phi_0)^-\|_{L^\infty(\X)} - 2K^-te^{2K^-T}\||\nabla\phi_0|\|^2_{L^\infty(\X)}
\]
belongs to \eqref{eq:spazi2}  and  is a weak (sub)solution of \eqref{eq:persupsub}. As $F_0 \geq G_0$ $\mm$-a.e., by Proposition \ref{pro:3} we conclude that $F_t \geq G_t$ $\mm$-a.e.\ for every $t \in [0,T]$ and in particular for $t:=T$.
\end{proof}

%

\section{Viscous approximation of the Hamilton--Jacobi semigroup}

Here we prove that on general $\RCD(K,\infty)$ spaces, solutions of the Hamilton--Jacobi--Bellman equation converge, when $\eps\downarrow0$, to solutions of the Hamilton--Jacobi equation. The starting point is the following simple corollary of Proposition \ref{pro:kyoto}.

\begin{Corollary}\label{cor:easy}
Let $(\X,\sfd,\mm)$ be an $\RCD(K,\infty)$ space with $K \in \R$, $\varphi \in \testi\X$ and set
\begin{equation}\label{eq:hjbsemigroup}
\varphi_t^\eps := \eps\log\h_{\eps t/2}(e^{\varphi/\eps}), \qquad t \geq 0,\, \eps > 0\,.
\end{equation}
Then for all $t > 0$ and $\eps \in (0,1)$ it holds
\begin{subequations}
\begin{align}
\label{eq:inftycontr-exp}
|\varphi_t^\eps|& \leq \|\varphi\|_{L^\infty(\X)}\,, \\
\label{eq:lipcontr-exp}
|\d\varphi^\eps_t|&\leq \Lip(\varphi_t^\eps) \leq e^{-K\eps t/2}\| |\nabla\varphi| \|_{L^\infty(\X)}\,,\\ 
\label{eq:lapcontr-exp}
(\Delta\varphi_t^\eps)^-& \leq \|(\Delta\varphi)^-\|_{L^\infty(\X)} + K^-te^{K^-\eps t}\| |\nabla\varphi| \|^2_{L^\infty(\X)}\,, 
\end{align}
\end{subequations}
where $K^- := \max\{0,-K\}$. Also, $t\mapsto\varphi^\eps_t$ is in $\AC_{\loc}([0,\infty),W^{1,2}(\X))$ with
\begin{equation}
\label{eq:HJBslow}
\frac\d{\d t}\varphi^\eps_t=\frac12|\nabla\varphi^\eps_t|^2+\frac\eps2\Delta\varphi^\eps_t\,.
\end{equation}
Furthermore, for every $B\subset\X$ bounded and $T>0$ we have
\begin{equation}\label{eq:lapcontr-int2}
\sup_{\substack{\eps\in(0,1) \\ t\in[0,T]}}\int_B |\Delta\varphi^\eps_t|\,\d\mm <\infty\,.
\end{equation}
\end{Corollary}

\begin{proof}
The estimates \eqref{eq:inftycontr-exp}, \eqref{eq:lipcontr-exp}, \eqref{eq:lapcontr-exp} are direct consequences of \eqref{eq:boundlinfty}, \eqref{eq:lipcontr}, \eqref{eq:lapcontr} respectively (for \eqref{eq:lipcontr-exp} we also use the Sobolev-to-Lipschitz property and \eqref{eq:C0reg}). Similarly, \eqref{eq:HJBslow} follows from \eqref{eq:hjb-eta2}. For the bound \eqref{eq:lapcontr-int2} we pick $\eta\in \Lip_{\bs}^+(\X)$ identically 1 on $B$ and note that
\[
\begin{split}
\int_B|\Delta\varphi^\eps_t|\,\d\mm \leq \int \eta|\Delta\varphi^\eps_t|\,\d\mm = \int -\nabla \eta\cdot\nabla\varphi^\eps_t + 2\eta(\Delta\varphi^\eps_t)^-\,\d\mm\,,
\end{split}
\]
so that the conclusion follows from the estimates \eqref{eq:lipcontr-exp}, \eqref{eq:lapcontr-exp}.
\end{proof}

Let us then recall the following simple lemma valid on general metric measure spaces, whose proof can be found in \cite{GigTam21}.

\begin{Lemma}\label{lem:dermiste}
Let $(\Y,\sfd_\Y,\mm_\Y)$ be a complete separable metric measure space endowed with a non-negative measure $\mm_\Y$ which is finite on bounded sets and assume that $W^{1,2}(\Y)$ is separable. Let $\ppi$ be a test plan and $f\in W^{1,2}(\Y)$. Then $t\mapsto \int f\circ\e_t\,\d\ppi$ is absolutely continuous and 
\[
\Big|\frac\d{\d t}\int f\circ\e_t\,\d\ppi \Big|\leq \int |\d f|(\gamma_t)|\dot\gamma_t|\,\d\ppi(\gamma)\qquad{\rm a.e.}\ t\in[0,1]\,,
\]
where the exceptional set can be chosen to be independent on $f$.

Moreover, if $(f_t)\in \AC([0,1],L^2(\Y))\cap L^\infty([0,1],W^{1,2}(\Y))$, then the map $t\mapsto\int f_t\circ\e_t\,\d\ppi$ is also absolutely continuous and 
\[
\frac\d{\d s}\Big(\int f_s\circ\e_s\,\d\ppi\Big)\restr{s=t}=\int \big(\frac\d{\d s}f_s\restr{s=t}\big)\circ\e_t\,\d\ppi+\frac\d{\d s}\Big(\int f_t\circ\e_s\,\d\ppi\Big)\restr{s=t}\qquad {\rm a.e.}\ t\in[0,1]\,.
\]
\end{Lemma}

We then have the following key result, whose proof is inspired by the one of \cite[Proposition 5.4]{GigTam21}.

%

\begin{Proposition}\label{pro:from-hjb-to-hopflax}
Let $(\X,\sfd,\mm)$ be an $\RCD(K,\infty)$ space with $K \in \R$ and $\varphi \in \testi\X$. Define  $\varphi_t^\eps$ as in \eqref{eq:hjbsemigroup}. Then 
\begin{equation}\label{eq:hjb-hopflax}
\lim_{\eps \downarrow 0}\varphi_t^\eps(x)  = \sup_{y \in \X}\Big\{\varphi(y) - \frac{\sfd^2(x,y)}{2t}\Big\}\qquad\forall t>0,\ x \in \X\,.
\end{equation}
Moreover, for every $t\geq 0$ and $\eta\in L^1\cap L^\infty\cap D(\Delta)$ non-negative with $\Delta\eta\in L^1(\X)$ we have
\begin{equation}
\label{eq:lapqt}
\int Q_t (-\varphi)\Delta\eta\,\d\mm\leq C(t,\varphi)\int \eta\,\d\mm\,,
\end{equation}
where
\begin{equation}
\label{eq:ctp}
C(t,\varphi):=\|(\Delta\varphi)^-\|_{L^\infty(\X)} + K^-t\| |\nabla\varphi| \|^2_{L^\infty(\X)}
\end{equation}
and $K^- := \max\{0,-K\}$, and
\begin{equation}
\label{eq:same}
\lip\, Q_t(-\varphi)=|\d Q_t(-\varphi)|\qquad \mm \times \mathcal L^1\textrm{-a.e. } (x,t) \in \X \times(0,+\infty)\,.
\end{equation}
\end{Proposition}

\begin{proof} By \eqref{eq:HJBslow},   \eqref{eq:lipcontr-exp} and  \eqref{eq:lapcontr-exp} we see that for every $T>0$ there is $C>0$ such that  $t\mapsto \varphi^\eps_t+Ct$ is increasing on $[0,T]$ for every $x\in\X$. Let $D \subset \X$ be a countable dense set. By Helly's Selection Principle applied to each of the functions $t \mapsto \varphi_t^\eps(x)+Ct$, with $x \in D$, and by a diagonal argument, there exists
\[
\lim_{k \to \infty}\varphi_t^{\eps_k}(x) =: \varphi_t(x), \qquad \forall x \in D,\ t\geq 0\,.
\]
From the uniform Lipschitz estimates \eqref{eq:lipcontr-exp} it is then clear that for every $t\geq 0$ the map $x\mapsto \varphi_t(x)$ can be uniquely extended to a Lipschitz function, still denoted by $\varphi_t$. Also, convergence on a dense set plus uniform continuity estimates easily give that 
\begin{equation}
\label{eq:epsk}
\lim_{k\to\infty}\varphi_t^{\eps_k}(x)=\varphi_t(x),\qquad\forall t\geq 0,\ x \in \X\,.
\end{equation}
With this said, \eqref{eq:hjb-hopflax} is proved if we are able to show that
\begin{equation}\label{eq:limitishopflax}
-\varphi_t = Q_t(-\varphi), \qquad \forall t > 0\,,
\end{equation}
where $Q_t$ is the Hopf--Lax semigroup defined in \eqref{eq:hli}. Indeed, if \eqref{eq:limitishopflax} holds, then $\varphi_t$ does not depend on the sequence $(\eps_k)$, so that the whole family $(\varphi_t^\eps)$ converges to $\varphi_t = -Q_t(-\varphi)$ pointwise as $\eps \to 0$, which is precisely \eqref{eq:hjb-hopflax}.

\noindent{\bf Inequality $\leq$ in \eqref{eq:limitishopflax}}. Fix $t>0$, pick $x,y \in \X$, $r>0 $, define
\[
\nu^r_x := \frac{1}{\mm(B_r(x))}\mm\restr{B_r(x)}\,, \qquad\qquad\qquad\nu^r_y := \frac{1}{\mm(B_r(y))}\mm\restr{B_r(y)}\,,
\]
and $\ppi^r$ as the lifting of the unique $W_2$-geodesic from $\nu^r_x$ to $\nu^r_y$ (recall point $(i)$ of Theorem \ref{thm:bm}). By Corollary \ref{cor:easy} we have that $s\mapsto\varphi^\eps_{st}$ is in $\AC([0,1],W^{1,2}(\X))$, thus we can apply the second part of Lemma \ref{lem:dermiste} to this function and $\ppi^r$ and get
\[
\begin{split}
\frac{\d}{\d s}\int {\varphi}^\eps_{st}\circ\e_s\,\d\ppi^r \geq \int t\frac{\d}{\d u}{\varphi}^\eps_u \restr{u={st}}(\gamma_s)-|\d {\varphi}^\eps_{st}|(\gamma_s)|\dot\gamma_s|\,\d\ppi^r(\gamma)\,.
\end{split}
\]
Now recall \eqref{eq:HJBslow}, so that the above and Young's inequality imply
\[
\frac{\d}{\d s}\int \varphi^\eps_{st}\circ\e_s\,\d\ppi^r\geq \int\frac{\eps t}2\Delta\varphi^\eps_{st}(\gamma_s)-\frac1{2t}|\dot\gamma_s|^2\,\d\ppi^r(\gamma)\,.
\]
Integrating in time, observing that $\varphi^\eps_0 = \varphi$ and recalling that $\ppi^r$ is optimal we get
\begin{equation}\label{eq:oasis}
\int\varphi^\eps_t\,\d\nu^r_y - \int\varphi^\eps\,\d\nu^r_x \geq -\frac1{2t}W_2^2(\nu^r_y,\nu^r_x) + \iint_0^1\frac{\eps t}2\Delta\varphi^\eps_{st}\circ\e_s\,\d s\,\d\ppi^r\,.
\end{equation}
Let $\eps=\eps_k\downarrow0$ be as in \eqref{eq:epsk} and combine the pointwise convergence with the uniform bound to deduce, via the dominated convergence theorem, that
\[
\lim_{k \to \infty} \int\varphi^{\eps_k}_t\,\d\nu^r_y = \int\varphi_t\,\d\nu^r_y\,.
\]
On the other hand, the uniform bound \eqref{eq:lapcontr-int2} and the fact that $\ppi^r$ has bounded compression and marginals uniformly supported in a bounded set (being the lifting of a $W_2$-geodesic between measures with bounded supports and densities) imply that the second term on the right-hand side in \eqref{eq:oasis} vanishes in the limit, whence
\[
\int\varphi_t\,\d\nu^r_y - \int\varphi\,\d\nu^r_x \geq -\frac1{2t}W_2^2(\nu^r_y,\nu^r_x)\,.
\]
Therefore, letting $r \downarrow 0$ we conclude from the arbitrariness of $x \in \X$ and the continuity of $\varphi_t:\X\to\R$ that
\begin{equation}
\label{eq:o1}
-\varphi_t(y) \leq Q_t(-\varphi)(y) \qquad \forall y\in \X\,.
\end{equation}

\noindent{\bf Inequality $\geq$ in \eqref{eq:limitishopflax}}. Fix $t>0$, define  $v_s^\eps := t\nabla\varphi_{(1-s)t}^\eps$ and note that the estimates  \eqref{eq:linftygradlap}, \eqref{eq:w12gradlap} and \eqref{eq:w12c} ensure that these vector fields satisfy the assumptions of Theorem \ref{thm:RLF}. We thus obtain existence of the Regular Lagrangian flow $F^\eps$. 

Fix $x\in\X$, $r>0$ and put $\ppi^\eps := \mm({B_{r}( {x})})^{-1}(F^\eps_\cdot)_*\mm\restr{B_{r}( {x})}$, where $F^\eps_\cdot:\X\to C([0,1],\X)$ is the $\mm$-a.e.\ defined map which sends $x$ to $s\mapsto F_s^\eps(x)$, and observe that the bound \eqref{eq:quantm}, identity \eqref{eq:quants} and the estimates \eqref{eq:lipcontr-exp}, \eqref{eq:lapcontr-exp} ensure that  $\ppi^\eps$ is a test plan and  the uniform bound
\begin{equation}
\label{eq:uniformtest}
(\e_s)_*\ppi^\eps \leq \mm({B_{r}( {x})})^{-1}e^{sC(t,\eps,\varphi)}\mm \qquad \forall s \in[0,1],\ \eps\in(0,1)\,,
\end{equation}
where 
\begin{equation}
\label{eq:defcteps}
C(t,\eps,\varphi):= \|(\Delta\varphi)^-\|_{L^\infty(\X)} + K^-te^{K^-\eps t}\| |\nabla\varphi| \|^2_{L^\infty(\X)}\,.
\end{equation}
As before, we apply Lemma \ref{lem:dermiste} to $\ppi^\eps$ and $s \mapsto {\varphi}^\eps_{(1-s)t}$ to obtain
\[
\begin{split}
\frac{\d}{\d s}\int \varphi^\eps_{(1-s)t}\circ\e_s\,\d\ppi^\eps
& = \int \Big(-t \frac{\d}{\d s'}\varphi^\eps_{s'}\restr{s'={(1-s)t}}\circ\e_s + \d{\varphi}^\eps_{(1-s)t}(v_s^\eps) \circ\e_s\Big)\,\d\ppi^\eps\\
& = \int\Big(-\frac{t}{2}|\d\varphi^\eps_{(1-s)t}|^2 - \frac{\eps t}{2}\Delta\varphi^\eps_{(1-s)t} + t|\d\varphi^\eps_{(1-s)t}|^2\Big)\circ\e_s\,\d\ppi^\eps\\
& = \int\Big( \frac{t}{2}|\d\varphi^\eps_{(1-s)t}|^2 - \frac{\eps t}{2}\Delta\varphi^\eps_{(1-s)t} \Big)\circ\e_s\,\d\ppi^\eps\,.
\end{split}
\]
Integrating in time and recalling \eqref{eq:quants} and that $\varphi^\eps_0 = \varphi$ we deduce
\begin{equation}
\label{eq:k6}
\int \varphi\circ\e_1 - \varphi_t^\eps\circ\e_0\,\d\ppi^\eps = \iint_0^1\frac1{2t}|\dot\gamma_s|^2 - \frac{\eps t}{2}\Delta\varphi^\eps_{(1-s)t}(\gamma_s)\,\d s\,\d\ppi^\eps(\gamma)\,.
\end{equation}
We claim that the family $(\ppi^\eps)_{\eps\in(0,1)}$ is tight. To see this we apply the tightness criterion Theorem \ref{thm:RLFstability} and note that \eqref{eq:stability1} holds because  
\begin{equation}
\label{eq:datoiniziale}
(\e_0)_*\ppi^\eps=\mm(B_r( x))^{-1}\mm\restr{B_r( x)}\qquad\forall \eps \in (0,1)\,.
\end{equation}
The first condition in \eqref{eq:stability2} follows from \eqref{eq:uniformtest} and the second one from the estimate \eqref{eq:lipcontr-exp}. Hence our claim is proved.

Now let $\eps=\eps_k\downarrow0$ be as in \eqref{eq:epsk} and pass to a non-relabeled subsequence to find a weak limit $\ppi$ of $(\ppi^{\eps_k})$. The weak convergence,   \eqref{eq:datoiniziale} and \eqref{eq:epsk} give
\[
\lim_{k\to\infty}\int \left(\varphi\circ\e_1 - \varphi_t^{\eps_k}\circ\e_0\right)\,\d\ppi^{\eps_k} = \int \left(\varphi\circ\e_1 - \varphi_t\circ\e_0\right)\,\d\ppi\,.
\]
On the other hand, the estimates \eqref{eq:lipcontr-exp}, and the identities \eqref{eq:datoiniziale}, \eqref{eq:quants} imply that the measures $(\e_s)_*\ppi^\eps$ are all concentrated on some bounded Borel set $B\subset\X$ for $\eps,s\in(0,1)$. Thus by  \eqref{eq:lapcontr-int2} we see that  the term with the Laplacian in \eqref{eq:k6} vanishes in the limit and thus taking into account the lower semicontinuity of the 2-energy we deduce that
\begin{equation}
\label{eq:perdopo}
\int \left(\varphi\circ\e_1- \varphi_t\circ\e_0\right)\,\d\ppi \geq \frac1{2t}\iint_0^1|\dot\gamma_s|^2\,\d s\,\d\ppi\geq \frac1{2t}\int\sfd^2(\gamma_0,\gamma_1)\,\d\ppi(\gamma)\,.
\end{equation}
Now note that \eqref{eq:o1} implies that 
\begin{equation}
\label{eq:hjcurve}
\frac{\sfd^2(\gamma_0,\gamma_1)}{2t}\geq \varphi(\gamma_1)-\varphi_t(\gamma_0)
\end{equation} 
for every curve $\gamma$, hence the above gives
\[
\int \left(\varphi\circ\e_1- \varphi_t\circ\e_0\right)\,\d\ppi \geq \frac1{2t}\int\sfd^2(\gamma_0,\gamma_1)\,\d\ppi(\gamma) \geq \int \left(\varphi\circ\e_1- \varphi_t\circ\e_0\right)\,\d\ppi\,,
\]
thus forcing the inequalities to be equalities. In particular, equality in \eqref{eq:hjcurve} holds for $\ppi$-a.e.\ $\gamma$ and since $(\e_0)_*\ppi = \mm\restr{B_{r}({x})}$, this is the same as to say that for $\mm$-a.e.\ $y \in B_{r}({x})$ there exists $x_y \in \supp((\e_1)_*\ppi)$ such that
\[
\frac{\sfd^2(x_y,y)}{2t} - \varphi(x_y) = -\varphi_t(y)\,.
\]
Combined with \eqref{eq:o1}, this implies that equality holds in \eqref{eq:o1} for $\mm$-a.e.\ $y \in B_{r}({x})$. Since both sides of $\eqref{eq:o1}$ are continuous in $y$, we deduce that equality holds for every $y \in B_{r}({x})$ and the arbitrariness of $r$ allows to conclude  that equality actually holds for every $y \in \X$.

\noindent{\bf Last claims}. Fix $\eta$ as in the assumptions and note that \eqref{eq:lapcontr-exp}, \eqref{eq:defcteps} and the trivial bound $a\geq -a^-$ valid for every $a \in \R$ give
\[
\int \varphi^\eps_t\Delta\eta\,\d\mm = \int \Delta\varphi^\eps_t\,\eta\,\d\mm \geq -\int (\Delta\varphi^\eps_t)^-\,\eta\,\d\mm \geq -C(t,\eps,\varphi) \int\eta\,\d\mm \qquad \forall\eps>0\,.
\]
Now we use the fact that $\Delta\eta\in L^1(\X)$, the pointwise convergence of $\varphi^\eps_t$ to $-Q_t(-\varphi)$ and the uniform bounds \eqref{eq:inftycontr-exp} to pass to the limit via dominated convergence and obtain \eqref{eq:lapqt}.

Passing to \eqref{eq:same}, we recall that the ``$\geq$'' inequality is always true, so we concentrate on the opposite one. We fix parameters $R,T,\tilde\eps>0$ and a point $x\in\X$. Then we pick $t\in(0,T]$, $\eps\in(0,1)$, consider as above the vector fields $v_s^\eps := t\nabla\varphi^\eps_{(1-s)t}$, their Regular Lagrangian Flows $F^\eps$ and then the plans $\ppi^\eps := \mm(B_{R-\tilde\eps}(x))^{-1}(F^\eps_\cdot)_*\mm\restr{B_{R-\tilde\eps}(x)}$. By the uniform estimates \eqref{eq:lipcontr-exp} there is $S = S(R,T,\tilde\eps)>0$ such that $\supp((\e_s)_*\ppi^\eps) \subset B_R(x)$ for every $s\in[0,S]$. Also, by \eqref{eq:uniformtest} and the fact that $C(t,\eps,\varphi)$ defined as in \eqref{eq:defcteps} is increasing in $t$, we see that possibly picking $S=S(R,T,\tilde\eps) > 0$ smaller we also have $(\e_s)_*\ppi^\eps \leq (1+\tilde\eps)\mm$ for every $s \in [0,S]$. We now let $\eps \downarrow 0$ as above to find a limit plan $\ppi$. The uniform estimate \eqref{eq:uniformtest} for $\ppi^\eps$ passes to the limit and ensures that
\begin{equation}
\label{eq:boundpiccolo}
(\e_s)_*\ppi \leq \frac{1+\tilde\eps}{\mm(B_{R-\tilde\eps}(x))}\mm\restr{B_R(x)}\qquad\forall s \in [0,S]\,.
\end{equation}
Also, the uniform Lipschitz estimate on the functions $\varphi_{st}$ (that follows from \eqref{eq:lipcontr-exp}), the fact that $\ppi$ is concentrated on equi-Lipschitz curves (being limit of plans with this property) and Proposition \ref{pro:11}  ensure that we can apply Lemma \ref{lem:dermiste} to differentiate  the map  $s\mapsto \int \varphi_{t(1-s)}\circ\e_s\,\d\ppi$, so that paying attention to the minus sign in \eqref{eq:limitishopflax} when applying Proposition \ref{pro:11} we get
\[
\begin{split}
\frac{\d}{\d s} \int \varphi_{t(1-s)}\circ\e_s\,\d\ppi
&\leq-\frac t2\int (\lip\,\varphi_{t(1-s)})^2\circ\e_s\,\d\ppi+\int |\d\varphi_{t(1-s)}|(\gamma_{s})|\dot\gamma_{s}| \,\d\ppi(\gamma)\\
&\leq \frac t2\int\Big(|\d\varphi_{t(1-s)}|^2- (\lip\,\varphi_{t(1-s)})^2\Big)\circ\e_s\,\d\ppi +\frac1{2t}\int |\dot\gamma_s|^2\,\d\ppi(\gamma)\,.
\end{split}
\]
Integrating on $[0,1]$ and recalling \eqref{eq:perdopo} we conclude that 
\[
\iint_0^1\Big(|\d\varphi_{t(1-s)}|^2- (\lip\,\varphi_{t(1-s)})^2\Big)\circ\e_s\,\d s\,\d\ppi \geq 0
\]
and since the integrand is almost surely non-positive, this forces
\[
|\d\varphi_{r}|=\lip\,\varphi_{r}\qquad \textrm{for a.e. } r\in[(1-S)t,t],\quad \mm\textrm{-a.e. on } \{\rho_{1-\frac rt}>0\}\,,
\]
where $\rho_s$ is the density of $(\e_s)_*\ppi$. Thus calling $A\subset\X\times\R^+$ the set (defined up to $\mm\times\mathcal L^1$-negligible sets) of $(x,t)$ such that $|\d\varphi_{t}|(x)=\lip\,\varphi_{t}(x)$ and using the trivial inequality $\mm(\{\rho>0\})\geq \|\rho\|_\infty^{-1}\int\rho\,\d\mm$ in conjunction with \eqref{eq:boundpiccolo}, we deduce that
\[
\mm\Big(A_r \cap B_R(x)\big)\Big) \geq \frac{\mm(B_{R-\tilde\eps}(x))}{1+\tilde\eps} \qquad \textrm{a.e. } r \in [(1-S)t,t]\,,
\]
where $A_r := \{x \in \X \,:\, (x,r) \in A\}$ is the $r$-section of $A$. We now pick $t:=(1-S)^nT$ in the above, $n\in\N$, to obtain that
\[
(\mm\times\mathcal L^1\restr{[0,T]})\Big(A\cap\big(B_R\times[0,T]\big)\Big) \geq T\frac{\mm(B_{R-\tilde\eps}(x))}{1+\tilde\eps}\,,
\]
thus letting $\tilde\eps\downarrow0$, the conclusion follows from the arbitrariness of $R,T$.
\end{proof}

\section{Large Deviations Principle}
\subsection{Heat kernel}\label{sec:LDPkernel}

Here we use Proposition \ref{pro:from-hjb-to-hopflax} to derive the Large Deviation Principle on \emph{proper} $\RCD(K,\infty)$ spaces. Let us recall that, from \cite[Theorem 2.5.28]{BBI01}, for an $\RCD(K,\infty)$ space being proper is equivalent to being locally compact. As it is clear from Theorems \ref{thm:lower-ldp}, \ref{thm:upper-ldp} all is left to do is to prove exponential tightness and generalize the results in Proposition \ref{pro:from-hjb-to-hopflax} to functions in $C_b(\X)$.

We shall make use of the following simple, yet crucial result, whose proof can be found in \cite[Lemma 6.7]{AmbrosioMondinoSavare13-2}.

\begin{Lemma}\label{le:testcutoff}
Let $(\X,\sfd,\mm)$ be a proper $\RCD(K,\infty)$ space with $K \in \R$. 

Then for every $K\subset U\subset\X$ with $K$ compact and $U$ open there is $\eta\in\testi\X$ with values in $[0,1]$, identically 1 on $K$ and with support in $U$.
\end{Lemma}

Such lemma and the monotonicity properties of the heat flow, the Cole--Hopf transform and the Hopf--Lax formula easily provide the following generalization of Proposition \ref{pro:from-hjb-to-hopflax}.

\begin{Proposition}\label{lem:ldp-lipschitz}
Let $(\X,\sfd,\mm)$ be a proper $\RCD(K,\infty)$ space with $K \in \R$ and $\varphi\in C_b(\X)$. Then \eqref{eq:hjb-hopflax} still holds true for all $t>0$ and $x \in \X$.
\end{Proposition}

\begin{proof}
Note that replacing $\varphi$ with $\varphi+c$ for $c\in\R$ we see that both sides of formula   \eqref{eq:hjb-hopflax} are increased by $c$. Thus, since $\varphi$ is bounded, we can assume that $\varphi$ is non-negative in proving the $\liminf$ inequality and that is non-positive in proving the $\limsup$.

Thus let $\varphi\in C_b(\X)$ be non-negative and note that a direct application of Lemma \ref{le:testcutoff} ensures that 
\[
\varphi=\sup\big\{\eta\in\testi\X\ :\ \eta\leq \varphi\text{ on $X$}\big\}\,.
\] 
Now fix $t,x$ and (by continuity and compactness) find $y\in\X$ realizing the sup in $\varphi(\cdot)-\frac{\sfd^2(\cdot,x)}{2t}$. Then find $(\eta_n)\subset\testi\X$ with $\eta_n\leq \varphi$ on $\X$ for every $n\in\N$ and $\eta_n(y)\to\varphi(y)$ as $n\to\infty$. For every $n\in\N$ we then have
\[
\begin{split}
\liminf_{\eps\downarrow0}\eps\log\int e^{\varphi/\eps}\hp_{\eps t/2}[x]\d\mm &\geq\liminf_{\eps\downarrow0}\eps\log\int e^{\eta_n/\eps}\hp_{\eps t/2}[x]\d\mm \\
\text{(by \eqref{eq:hjb-hopflax})}\qquad&= \sup_{z\in \X}\Big\{\eta_n(z) - \frac{\sfd^2(x,z)}{2t}\Big\}\geq \eta_n(y) - \frac{\sfd^2(x,y)}{2t}\,.
\end{split}
\]
Letting $n\to\infty$ we see that the $\liminf$ inequality in  \eqref{eq:hjb-hopflax} holds for the given function $\varphi$. The $\limsup$ follows along the same lines.
\end{proof}
As a direct consequence we obtain:

\begin{Theorem}[Heat kernel LDP]\label{thm:ldp-heatkernel}
Let $(\X,\sfd,\mm)$ be a proper $\RCD(K,\infty)$ space with $K \in \R$. For $x\in\X$ and  $t>0$ we consider the measures $\mu_t[x] := \hp_{t}[x]\mm=\h_t\delta_x$. Then the family $(\mu_t[x])_{t>0}$ satisfies a Large Deviations Principle with speed $t^{-1}$ and rate
\[
I(z) := \frac{\sfd^2(x,z)}{4}\,.
\]
\end{Theorem}
\begin{proof} Since we assumed $\X$ to be proper, the rate function $I$ has compact sublevels. Hence Theorems \ref{thm:lower-ldp}, \ref{thm:upper-ldp} apply and the conclusion follows from Proposition \ref{lem:ldp-lipschitz} above (we are applying \eqref{eq:hjb-hopflax} with $\eps:=t$ and $t:=2$).
\end{proof}

\subsection{Brownian motion}\label{sec:LDPbrownian}
We now prove the LDP for the Brownian motion starting from a given point $\bar x$ of a proper $\RCD(K,\infty)$ space. The Brownian motion is the continuous Markov process with transition probabilities given by the heat kernel measures and for the purpose of the present discussion we identify the Brownian motion starting from $\bar x$ with its law, i.e., the Borel probability measure over $C([0,1],\X)$, characterized as the only one concentrated on curves starting from $\bar x$, with the Markov property and transition probabilities given by the heat kernel measures $\mu_{s-t}[x]$ (see \cite[Theorem 6.8]{AmbrosioGigliSavare11-2} and \cite[Theorem 7.5]{AmbrosioGigliMondinoRajala15} for existence and uniqueness). In other words $\B$ is characterized by the validity of 
\begin{equation}\label{eq:chapman-kolmogorov}
(\e_{t_1},...,\e_{t_n})_*\B = \hp_{t_1}[\bar x](x_1)\hp_{t_2-t_1}[x_1](x_2) \cdots \hp_{t_n - t_{n-1}}[x_{n-1}](x_n)\d\mm(x_1) \cdots \d\mm(x_n)\,,
\end{equation}
for every $0<t_1<\ldots t_n\leq 1$. 

A remark on the existence and uniqueness of the Brownian motion starting at \emph{any} arbitrary point $\bar x \in \X$ is in order: \cite[Theorem 6.8]{AmbrosioGigliSavare11-2} and \cite[Theorem 7.5]{AmbrosioGigliMondinoRajala15} actually ensure the existence of the Brownian motion starting at any $\bar x$, \emph{up to an $\mm$-negligible set}. However, given that $\X$ is assumed to be proper, by the Feller property one can show that the Brownian motion can actually start at any $\bar x \in \X$. As for the uniqueness, this should also be understood up to $\mm$-equivalence, but the continuity of $\hp_t$ entails the validity of \eqref{eq:chapman-kolmogorov} for every $\bar x \in \X$ rather than just $\mm$-a.e.\ $\bar x \in \X$.

Then the slowed-down Brownian motion $\tilde\B_t\in\prob{C([0,1],\X)}$ is defined as
\[
\tilde\B_t := ({\rm restr}_0^t)_*\B \qquad \textrm{where} \quad ({\rm restr}_0^t)(\gamma)(s) := \gamma_{st},\, \gamma \in C([0,1],\X)\,.
\]
We will show that, as expected, $(\tilde\B_t)$ satisfies a LDP with  rate function $I:C([0,1],\X) \to [0,\infty]$ defined as 
\[
I(\gamma) := \left\{\begin{array}{ll}
\displaystyle{\frac{1}{4}\int_0^1 |\dot{\gamma}_t|^2\,\d t} & \qquad \text{ if } \gamma \in \AC^2([0,1],\X),\,\gamma_0 = \bar x, \\
+\infty & \qquad \text{ otherwise.}
\end{array}\right.
\]
Note that a bound on the kinetic energy provides a ($\tfrac12$-H\"older) continuity estimate; thus since $\X$ is a proper space and $I(\gamma)=\infty$ if $\gamma_0\neq\bar x$, it is easy to see that sublevels of $I$ are compact.

Our proof of the LDP for the Brownian motion will be based on the representation formula
\begin{equation}
\label{eq:reprI}
I(\gamma)=\frac14\sup\sum_{i=0}^{n-1}\frac{\sfd^2(\gamma_{t_{i+1}},\gamma_{t_i})}{|t_{i+1}-t_i|}\qquad\forall\gamma\in C([0,1],\X)\ \text{ with }\gamma_0=\bar x\,,
\end{equation}
where the supremum is taken among all $n\in\N$ and partitions $0=t_0<\cdots<t_n=1$ of $[0,1]$. A proof of \eqref{eq:reprI} can be found, e.g., in \cite[Lemma 2.5]{Lisini07} (and its proof). Actually, it is easy to see that \eqref{eq:reprI} holds even if we only take the supremum over partitions such that $\sfd(\gamma_{t_{i+1}}, \gamma_{t_i}) > 0$ for all $i=0,\dots,n-1$. 

%
%
%
%

We start with the following auxiliary result, that --starting from the LDP for the heat flow-- provides LD lower and upper bounds for the family of measures obtained by integrating the $\mu_t[x]$'s as $x$ varies in some fixed compact set.

\begin{Lemma}\label{pro:ldp-superposition}
Let $(\X,\sfd,\mm)$ be a proper $\RCD(K,\infty)$ space with $K \in \R$, $B \subset \X$ compact and $(\sigma_t)_{t > 0} \subset \prob\X$ with $\sigma_t$ concentrated on $B$ for every $t > 0$. For every $t>0$ and $x \in \X$ define $\overline{\mu}_t\in\prob\X$ as
\[
\overline{\mu}_t(E) := \int \mu_t[x](E)\d\sigma_t(x), \qquad \forall E \textrm{ Borel}.
\]
Then for every $\varphi\in C_b(\X)$ we have
\begin{equation}
\label{eq:lemmino}
\begin{split}
\limsup_{t \downarrow 0} t\log\int e^{\varphi/t}\,\d\overline{\mu}_t &\leq \sup_{y \in \X}\Big\{\varphi(y) - \frac{\sfd_-^2(y,B)}{4}\Big\}\,,\\
\liminf_{t \downarrow 0} t\log\int e^{\varphi/t}\,\d\overline{\mu}_t &\geq \sup_{y \in \X}\Big\{\varphi(y) - \frac{\sfd_+^2(y,B)}{4}\Big\}\,,
\end{split}
\end{equation}
where $\sfd_-(x,B) := \min_{y \in B}\sfd(x,y)$ and $\sfd_+(x,B) := \max_{y \in B}\sfd(x,y)$.
\end{Lemma}

\begin{proof} With the same approximation arguments based on monotonicity used to prove Proposition \ref{lem:ldp-lipschitz}, we can, and will, reduce to the case $\varphi\in\testi\X$.  Fix such $\varphi$, set $\varphi_t := \varphi_2^t$ with $\varphi_{t=2}^{\eps=t}$ defined in \eqref{eq:hjbsemigroup} and recall that by \eqref{eq:lipcontr-exp} we have the uniform estimate
\begin{equation}
\label{eq:lipphit}
L:=\sup_{t\in(0,1)}\Lip(\varphi_t)<\infty\,.
\end{equation}
Note that 
\[
t\log\int e^{\varphi/t}\,\d\overline{\mu}_t  = t\log\int\Big(\int e^{\varphi/t}\, \d\mu_t[x]\Big)\d\sigma_t(x) = t\log\int e^{\varphi_t/t}\,\d\sigma_t 
\]
and thus
\[
\begin{split}
t\log\int e^{\varphi/t}\,\d\overline{\mu}_t \qquad
\left\{\begin{array}{l}
\displaystyle{\leq t\log\Big(\int\exp\Big(\frac{\sup_{B}\varphi_t}{t}\Big)\d\sigma_t\Big) = \sup_B \varphi_t=\varphi_t(x_t)} \\
\ \\
\displaystyle{\geq t\log\Big(\int\exp\Big(\frac{\inf_{B}\varphi_t}{t}\Big)\d\sigma_t\Big) = \inf_B \varphi_t=\varphi_t(x'_t)}
\end{array} \right.
\end{split}
\]
for some well-chosen $x_t,x_t'\in B$. Now let $t_n\downarrow0$ be realizing $\limsup_{t\downarrow0}\sup_B \varphi_t$ and, by compactness of $B$, extract a further subsequence, not relabeled, such that $x_{t_n}\to x_0\in B$.

Then keeping \eqref{eq:lipphit} in mind we have $\sup_B\varphi_{t_n}=\varphi_{t_n}(x_{t_n})\leq L\sfd(x_{t_n},x_0)+\varphi_{t_n}(x_0)$ and letting $n\to\infty$ and recalling \eqref{eq:hjb-hopflax} we get the first inequality in \eqref{eq:lemmino}. The second inequality follows along similar lines.
\end{proof}

We are now in the position to prove the LDP for the Brownian motion.

\begin{Theorem}[LDP for Brownian motion]\label{thm:brownian-upper}
Let $(\X,\sfd,\mm)$ be a proper $\RCD(K,\infty)$ space with $K \in \R$, $\bar x \in \X$ and $(\tilde\B_t)\subset \prob{C([0,1],\X)}$, $I:C([0,1],\X) \to [0,\infty]$ be defined as in the beginning of the section.

Then the family $(\tilde\B_t)_{t>0}$ satisfies a Large Deviation Principle with speed $t^{-1}$ and rate function $I$.
\end{Theorem}
\begin{proof}\ \\

\noindent{\textbf{Preliminary considerations}}. We shall use the equivalence between $(i)$ and $(iii')$ in Theorems \ref{thm:lower-ldp},  \ref{thm:upper-ldp}. Let $\bar \gamma\in C([0,1],\X)$ and note that if it is constant, the validity of $(iii')$ in both these theorems is obvious: the $\Gamma$-$\liminf$ follows from the non-negativity of the entropy, while for the $\Gamma$-$\limsup$ we pick $\tilde\B_t$ as recovery sequence (since $\B$ is concentrated on curves starting from $\bar x$, the $\tilde\B_t$'s weakly converge to the curve constantly equal to $\bar x$). Thus we assume that $\bar\gamma$ is not constant and, without loss of generality, pick $n \in \N$ and $0 = t_0 < \ldots < t_n \le 1$ such that $\sfd(\bar\gamma_{t_{i+1}}, \bar \gamma_{t_i}) > 0$ for every $i$ (recall that \eqref{eq:reprI} holds even if we only take the supremum over partitions with this property). Then for  $r>0$ we define 
\[
U_{r,i} := \{\gamma \,:\, \gamma_{t_i}\in B_{r(i+1)}(\bar \gamma_{t_i})\}, \quad i=0,\ldots,n \qquad \textrm{and} \qquad U_r^n := \bigcap_{i=0}^n U_{r,i}
\]
and note that these are open sets in $C([0,1],\X)$. We let 
\[
\ell:=\frac14\sum_{i=0}^{n-1}  \frac{\sfd^2(\bar\gamma_{t_{i+1}},\bar\gamma_{t_i})}{t_{i+1} - t_{i}}>0
\]
and claim that for some $C>0$ depending on $\bar\gamma,n,(t_i)$ the bounds
\begin{subequations}
\label{eq:claimll}
\begin{align}
\label{eq:claimlimi}
\liminf_{t\downarrow0}t\log(\tilde\B_t(U_r^n))&\geq -\ell+Cr\,,\\
\label{eq:claimlims}
\limsup_{t\downarrow0}t\log(\tilde\B_t(\overline U_r^n))&\leq -\ell+Cr\,,
\end{align}
\end{subequations}
hold for every $r>0$ sufficiently small. To see this we argue by induction over $n \ge 1$.  Recalling that the Brownian motion has the Markov property it holds, for every $n \ge 1$,
\begin{equation}\label{eq:factorization}
\tilde\B_t(U_r^n) = \tilde\B_t \left( U_{r}^{n-1} \right ) \int \mu_{t (t_n-t_{n-1})}[x_{n-1}]( U_{r, n} ) \,\d\sigma_t(x_{n-1})
\end{equation}
where, for $E\subseteq \X$ Borel,
\[
\sigma_t(E) = \frac{\tilde\B_t\left( \left\lbrace \gamma\, : \, \gamma_{t_{n-1}} \in E \right\rbrace \cap U^{n-1}_r \right)}{\tilde\B_t \left(  U_r^{n-1}\right )}\,.
\]
Note that if $t$ is sufficiently small, then the definition is meaningful, because $\tilde\B_t \left( U^{n-1} \right)>0$. Indeed, if $n=1$ we have $\tilde\B_t \left(  U_r^{0}\right ) = 1$ because $\B$ starts at $\bar{x}$, while if $n >1$ we have by inductive assumption that \eqref{eq:claimlimi} holds, with $n-1$ instead of $n$.

Starting from \eqref{eq:factorization} and arguing by induction, we see that \eqref{eq:claimlimi} follows if we show that the bound
\begin{equation}
\label{eq:perclms}
\liminf_{t\downarrow0} t\log\left( \int \mu_{t (t_n-t_{n-1})}[x_{n-1}]( U_{r, n} ) \,\d\sigma_t(x_{n-1}) )\right) \geq -\frac14 \frac{\sfd^2(\bar\gamma_{t_n},\bar\gamma_{t_{n-1}})}{t_{n} - t_{n-1}} + Cr
\end{equation}
holds for every $r>0$ sufficiently small, where $C$ depends on $\bar{\gamma}$, $n$ and $(t_i)_{i=1}^n$. Since the measure $\sigma_t$ is by construction concentrated on $\overline B_{rn}(\bar\gamma_{t_{n-1}})$, the bound \eqref{eq:perclms} follows from Theorem \ref{thm:lower-ldp} and Lemma \ref{pro:ldp-superposition} above applied with $\sigma_t$'s just defined and $B := \overline B_{rn}(\bar\gamma_{t_{n-1}})$. More precisely: by Theorem \ref{thm:lower-ldp} the second inequality in \eqref{eq:lemmino} implies that the measures $\int\mu_{t(t_{n}-t_{n-1})}[x]\,\d \sigma_t$ satisfy a LDP lower bound with rate function $I := \frac14 \sfd_+^2(\cdot, \overline B_{rn}(\bar\gamma_{t_{n-1}}))$, so that if we apply the definition of LDP lower bound to the open set $O := B_{r(n+1)}(\bar \gamma_{t_{n}})$ we obtain
\begin{equation}\label{eq:intermezzo}
\liminf_{t\downarrow0} t\log\left(\int \mu_{t (t_n-t_{n-1})}[x_{n-1}]( U_{r, n} ) \,\d\sigma_t(x_{n-1}) \right) \geq -\inf_O I\,;
\end{equation}
then note that $\sfd_+(x, \overline B_{rn}(\bar\gamma_{t_n})) \leq \sfd(x,\bar\gamma_{t_n}) + rn$,
whence
\[
\inf_{x \in O}\sfd_+(x, \overline B_{rn}(\bar\gamma_{t_n})) \leq \big(\sfd(\bar\gamma_{t_{n}},\bar\gamma_{t_{n-1}}) + rn - r(n+1)\big)^+ = \big(\sfd(\bar\gamma_{t_{n}},\bar\gamma_{t_{n-1}}) - r\big)^+\,,
\]
and this yields
\[
\inf_O I \leq \frac14 \big((\sfd(\bar\gamma_{t_{n}}, \bar\gamma_{t_{n-1}})-r)^+\big)^2 \leq \frac14 \sfd^2(\bar\gamma_{t_{n}}, \bar\gamma_{t_{n-1}})-Cr
\]
for every $r>0$ sufficiently small and some $C$, because $\sfd^2(\bar\gamma_{t_{n}}, \bar\gamma_{t_{n-1}})>0$. Plugging this inequality into \eqref{eq:intermezzo}, we conclude that \eqref{eq:perclms} holds.
 
The proof of \eqref{eq:claimlims} follows the very same lines replacing each $U_{r,n}$ with $\overline{U}_{r,n}$ and accordingly defining $\bar{\sigma}_t$ in order to write the suitable analog of \eqref{eq:factorization}. Indeed, we get
\[
\limsup_{t\downarrow0} t\log\left(\int \mu_{t (t_n-t_{n-1})}[x_{n-1}]( \overline{U}_{r, n} ) \,\d \bar{\sigma}_t(x_{n-1}) \right) \leq -\inf_D I\,,
\]
where $I := \frac14 \sfd_-^2(\cdot, \overline B_{rn)}(\bar\gamma_{t_{n-1}}))$ and $D := \overline B_{r(n+1)}(\bar \gamma_{t_{n}})$. At this point, one should observe that $\sfd(x,z) \geq \sfd(x,\bar\gamma_{t_{n-1}}) - \sfd(z,\bar\gamma_{t_{n-1}})$, which implies $\sfd_-(x,B_{rn}(\bar\gamma_{t_{n-1}})) \geq \sfd(x,\bar\gamma_{t_{n-1}}) - rn$ and therefore $\inf_D \sfd_-(x,B_{rn}(\bar\gamma_{t_{n-1}})) \geq (\sfd(\bar\gamma_{t_{n}},\bar\gamma_{t_{n-1}}) - r(2n+1))^+$. This allows us to rewrite the inequality above as
\[
\limsup_{t\downarrow0} t\log\left(\int \mu_{t (t_n-t_{n-1})}[x_{n-1}]( \overline{U}_{r, n} ) \,\d \bar \sigma_t(x_{n-1}))\right) \leq -\frac14 \frac{\sfd^2(\bar\gamma_{t_n},\bar\gamma_{t_{n-1}})}{t_{n} - t_{n-1}} + Cr\,,
\]
whence \eqref{eq:claimlims}, by induction.
 
\noindent{\textbf{Upper bound}}.
Let $(\nu_t)\subset \prob{C([0,1],\X)}$ be weakly converging to $\delta_{\bar \gamma}$ and, with the same notation as above, note that $\bar\gamma$ belongs to the open set $U_r$. Hence $a=a_{t,r}:=\nu_t(U_r)\to1$. Now condition both $\nu_t$ and $\tilde\B_t$ to $U_r$, i.e.\ introduce the measures
\begin{equation}
\label{eq:condB}
\hat\nu_{t,r} := a^{-1}\nu_t\restr{U_r}\qquad\text{ and }\qquad \hat\B_{t,r} := b^{-1}\tilde\B_t\restr{U_r},\quad\text{ where }b=b_{t,r} := \tilde\B_t(U_r)
\end{equation}
and let $\rho_t:=\frac{\d\nu_t}{\d\tilde\B_t}$ (if $\nu_t \not\ll \tilde\B_t$ then $H(\nu_t\,|\,\tilde\B_t) = +\infty$ so that $\nu_t$ can be disregarded in studying the LDP upper bound). Also, put $u(z) := z\log z$ and note that direct (and simple) algebraic manipulations in conjunction with the convexity of $u$ and Jensen's inequality give
\[
\begin{split}
H(\nu_t\,|\,\tilde\B_t) & = H(\hat\nu_{t,r}\,|\,\hat\B_{t,r}) + b(1-a^{-1})\fint_{U_r}u(\rho_t)\,\d\tilde\B_t + (1-b)\fint_{U_r^c}u(\rho_t) \,\d\tilde\B_t - \log(\tfrac{b}{a})\\
& \geq b(1-a^{-1})u\Big(\underbrace{\fint_{U_r} \rho_t\,\d\tilde\B_t}_{=\frac{a}{b}}\Big)+(1-b)u\Big(\underbrace{\fint_{U_r^c}\rho_t \,\d\tilde\B_t}_{=\frac{1-a}{1-b}}\Big)-\log(\tfrac{b}{a})\\
&=a\big(\log a-\log b\big)+(1-a)\big(\log(1-a)-\log(1-b)\big)\,.
\end{split}
\]
Since we know that as $t\downarrow0$ we have $a=a_{t,r}\to  1$ and, by \eqref{eq:claimlims}, that $b=b_{t,r}\to 0$ if $r$ is sufficiently small (as $\ell>0$), recalling also  \eqref{eq:claimlims} we deduce that
\[
\liminf_{t\downarrow0}tH(\nu_t\,|\,\tilde\B_t)\geq \ell-Cr\,.
\]
This bound is true for every $r>0$ and every partition of $[0,1]$, hence first letting $r\downarrow0$ and then taking the supremum over all partitions, we conclude  recalling \eqref{eq:reprI}  and  Theorem \ref{thm:upper-ldp}.

\noindent{\textbf{Lower bound}}. With the same notation as in the beginning of the proof, we   let $K\subset C([0,1],\X)$ be the set of curves  $\gamma$ that are geodesic on the intervals $[t_i,t_{i+1}]$ and such that $\gamma_{t_i}=\bar\gamma_{t_i}$ for every $i=0,\ldots,n$. Then clearly $K$ is non-empty (as $\X$ is geodesic), is contained in  $U_r$ for every $r>0$ and  $I(\gamma)=\ell$ for every $\gamma\in K$. Since $\X$ is proper, $K$ is also compact.

Now let $V \subset C([0,1],\X)$ be an open neighbourhood of  $K$. We claim that for every  $r>0$ sufficiently small we have
\[
\alpha:=\inf_{\overline U_r\setminus V}I>\ell-Cr\,, 
\]
where $C$ is the constant in \eqref{eq:claimll}. Indeed, if this were not the case we would have curves $\gamma_r\in \overline U_r\setminus V$ with $\liminf_rI(\gamma_r)\leq \ell$. By the compactness of the sublevels of $I$, there would be a sequence $r_j\downarrow0$ realizing the $\liminf_rI(\gamma_r)$ converging to some  $\gamma$. Then $\gamma$ should be outside $V$, such that $\gamma_{t_i}=\bar\gamma_{t_i}$ for every $i=0,\ldots,n$ and with $I(\gamma)=\ell$. It is readily verified that this last two conditions imply $\gamma\in K$, which however contradicts $\gamma\notin V$.

Hence our claim is true and by the LD upper bound that we already proved we have
\[
\limsup_{t\downarrow0}t\log(\tilde\B_t(\overline U_r\setminus V))\leq -\alpha\,.
\]
Considering the conditioning $\hat\B_{t,r}$ of $\tilde\B_t$ as introduced before in \eqref{eq:condB}, this latter bound and \eqref{eq:claimlimi} give
\[
\limsup_{t\downarrow0}t\log(\hat\B_{t,r}(V^c)) \leq \limsup_{t\downarrow0} t\log\Big(\frac{\tilde\B_t(\overline U_r\setminus V)}{\tilde\B_t(U_r)}\Big)\leq \ell-Cr-\alpha<0\,,
\]
thus forcing $\hat\B_{t,r}(V^c) \to 0$. In summary, using the very definition of $\hat\B_{t,r}$ first and \eqref{eq:claimlimi} then, we have that
\[
\limsup_{r\downarrow0}\limsup_{t\downarrow0}t H(\hat\B_{t,r}\,|\,\tilde\B_t) = -\liminf_{r\downarrow0}\liminf_{t\downarrow0}t \log (\tilde\B_t(U_r)) \leq \ell \leq I(\bar \gamma)
\]
and that for every $r>0$ sufficiently small the mass of $\hat\B_{t,r}$ tends to be concentrated, as $t\downarrow0$, in the neighbourhood $V$ of $K$. Since this is true for every such neighbourhood $V$ and $K$ is compact, by a diagonal argument we can find $r(t)$ going to 0 sufficiently slow so that the family $(\hat\B_{t,r(t)})_{t\in(0,1)}$ satisfies
\begin{equation}
\label{eq:quasi}
\limsup_{t\downarrow0}t H(\hat\B_{t,r(t)}\,|\,\tilde\B_t)\leq I(\gamma)\,,
\end{equation}
is tight and every weak limit is concentrated on $K$. 

Now we consider a sequence of partitions indexed by $j$ getting finer and finer: it is clear from the continuity of $\bar\gamma$ that the corresponding compact sets $K_j$ converge to $\{\bar\gamma\}$ in the Hausdorff distance. For each of these partitions we have a corresponding family $(\hat B_{t,r(t,j),j})_{t>0}$ satisfying \eqref{eq:quasi} with cluster points supported on $K_j$. Hence by a further diagonalization we can find $j(t)$ going to $+\infty$ so slowly that 
\[
\limsup_{t\downarrow0}t H(\hat\B_{t,r(t,j),j(t)}\,|\,\tilde\B_t) \leq I(\gamma)\,,
\]
and since the construction forces the weak convergence of $\hat\B_{t,r(t,j),j(t)}$ to $\delta_{\bar \gamma}$, the conclusion follows by Theorem \ref{thm:lower-ldp}.
\end{proof}

\section{On the convergence of the Schr\"odinger problem to the Monge-Kantorovich one}

As a consequence of Theorems \ref{thm:ldp-heatkernel} and \ref{thm:brownian-upper}, we can prove that, also in the setting of proper $\RCD(K,\infty)$ spaces, the Schr\"odinger problem $\Gamma$-converges to the quadratic optimal transport problem, which seems new. In order to state this fact in a rigorous way, let us first introduce the measures $\sfR^\eps \in \prob{C([0,1],\X)}$ and $\sfR_{0\eps} \in \prob{\X^2}$ as
\[
\sfR^\eps(\d\gamma) := \int\sfB_{\eps/2}[x](\d\gamma)\,\d\mu_0(x) \quad \textrm{and} \quad \sfR_{0\eps}(\d x\d y) := (\e_0,\e_1)_* \sfR^\eps = \hp_{\eps/2}[x](y)\mm(\d y)\mu_0(\d x)\,,
\]
where $\sfB_{\eps/2}[x]$ denotes the slowed-down Brownian motion starting at $x$ and $\mu_0 \in \prob\X$ is arbitrary but fixed. Let us also define the \emph{entropic cost}, in static and dynamic form:
\[
\mathscr{C}_\eps(\mu_0,\nu) := \inf H(\ggamma\,|\,\sfR_{0\eps}) \qquad \textrm{and} \qquad \mathscr{C}'_\eps(\mu_0,\nu) := \inf H(Q\,|\,\sfR^\eps),
\]
where in the first case the infimum runs over all couplings $\ggamma \in \prob{\X^2}$ between $\mu_0$ and $\nu \in \prob\X$, while in the second case over all measures $Q \in \prob{C([0,1],\X)}$ with $(\e_0)_*Q=\mu_0$ and $(\e_1)_*Q = \nu$. The minimization problems written above are the so-called static and dynamic form of the Schr\"odinger problem (the reader is referred to \cite{Leonard14} for more details on the topic). With this premise, we have the following

\begin{Corollary}\label{cor:schrod-to-wasser}
Let $(\X,\sfd,\mm)$ be a proper $\RCD(K,\infty)$ space with $K \in \R$. Then
\[
\Gamma\textrm{-}\lim_{\eps \downarrow 0} \eps H(\cdot\,|\,\sfR_{0\eps}) + \iota_0 = \int \sfd^2(x,y)\,\d \cdot + \iota_0 \qquad \textrm{in } \prob{\X^2}\,,
\]
where $\iota_0$ is the indicator function of the marginal constraint $\ggamma_0 = \mu_0$ (here $\ggamma_0$ denotes the projection of $\ggamma$ on the first factor), and
\[
\Gamma\textrm{-}\lim_{\eps \downarrow 0} \eps H(\cdot\,|\,\sfR^\eps) + \iota'_0 = \iint_0^1 |\dot{\gamma}_t|^2\,\d t\d \cdot + \iota'_0 \qquad \textrm{in } \prob{C([0,1],\X)}\,,
\]
where $\iota'_0$ is the indicator function of the marginal constraint $(\e_0)_*Q = \mu_0$.

In particular, for every $\mu_1 \in \prob\X$ there exists $(\mu_1^\eps) \subset \prob\X$ with $\mu_1^\eps \rightharpoonup \mu_1$ as $\eps \downarrow 0$ such that
\begin{equation}\label{eq:schrod-to-wasser}
\lim_{\eps \downarrow 0} \eps\mathscr{C}_\eps(\mu_0,\mu_1^\eps) = \lim_{\eps \downarrow 0} \eps\mathscr{C}'_\eps(\mu_0,\mu_1^\eps) = \frac{1}{2} W_2^2(\mu_0,\mu_1)\,.
\end{equation}
\end{Corollary}

\begin{proof}
It is a general fact that $\mathscr{C}_\eps(\mu_0,\nu) = \mathscr{C}'_\eps(\mu_0,\nu)$, see for instance \cite[Proposition 2.3]{Leonard14}. The statement about the static Schr\"odinger problem follows then from Theorem \ref{thm:ldp-heatkernel}  and \cite[Proposition 3.1, Corollary 3.2 and Theorem 3.3]{Leonard12}, while the one on the dynamical Schr\"odinger problem from Theorem \ref{thm:brownian-upper} and \cite[Proposition 3.4, Corollary 3.5 and Theorem 3.6]{Leonard12}.
\end{proof}

A tightly-linked result has recently been obtained by the second-named author in \cite{MTV20}, as a byproduct of a more general and abstract $\Gamma$-convergence statement. In a few words, in \cite[Theorem 4.6 and Section 6.1]{MTV20} it is shown that on a proper $\RCD(K,\infty)$ space $(\X,\sfd,\mm)$ for all $\mu_0,\mu_1 \in \probt\X$ there exist $\mu_0^\eps \rightharpoonup \mu_0$ and $\mu_1^\eps \rightharpoonup \mu_1$ such that
\begin{equation}\label{eq:dynamical-to-wasser}
\lim_{\eps \downarrow 0} \eps\tilde{\mathscr{C}}_\eps(\mu_0^\eps,\mu_1^\eps) = \frac{1}{2} W_2^2(\mu_0,\mu_1)\,,
\end{equation}
where
\begin{equation}\label{eq:bbs}
\tilde{\mathscr{C}}_\eps(\mu_0,\mu_1) := \inf\Big\{\frac{1}{2\eps}\iint_0^1 |v_t|^2\rho_t\,\dt\d\mm + \frac{\eps}{8}\iint_0^1 |\nabla\log\rho_t|^2\rho_t\,\dt\d\mm\Big\}
\end{equation}
and the infimum runs over all couples $(\mu_t,v_t)$, $\mu_t=\rho_t\mm$, solving the continuity equation $\partial\mu_t + {\rm div}(v_t\mu_t) = 0$ in the weak sense and with $\mu_0,\mu_1$ as initial and final constraints. Moreover, $\mu_0^\eps$ can be chosen equal to $\mu_0$ for all $\eps > 0$ provided $H(\mu_0\,|\,\mm) < \infty$ (analogously for $\mu_1^\eps$), see \cite[Theorem 4.3, Corollary 4.4 and Section 6.1]{MTV20}

After \cite{GigTam20}, $\tilde{\mathscr{C}}_\eps = \mathscr{C}_\eps$ on $\RCD^*(K,N)$ spaces with $N<\infty$. Therefore \eqref{eq:schrod-to-wasser} and \eqref{eq:dynamical-to-wasser} are equivalent and Corollary \ref{cor:schrod-to-wasser} provides an alternative proof of them. If $N=\infty$ instead, the Schr\"odinger problem defined above and its representation \`a la Benamou--Brenier \eqref{eq:bbs} are not known to be equivalent. Hence, given the present state of the art, Corollary \ref{cor:schrod-to-wasser} and in particular \eqref{eq:schrod-to-wasser} are new and the latter, although heuristically equivalent, is independent of \eqref{eq:dynamical-to-wasser}.

\section{Declarations}
\noindent{\bf Acknowledgements:} The authors would like to thank the anonymous referee for the detailed report and the fruitful suggestions.

\noindent{\bf Ethical approval:} Not applicable.

\noindent{\bf Competing Interests:} The authors have no competing interests of financial or personal nature.

\noindent{\bf Authors' contributions:} All authors made equal contributions to the preparation of this manuscript.

\noindent{\bf Funding:} LT author gratefully acknowledges support by the European Union through the ERC-AdG ``RicciBounds", PI Prof.\ K.-T.\ Sturm, since most of this work was written while he was a member of the Institut f\"ur Angewandte Mathematik of the University of Bonn. DT has been supported by the HPC Italian National Centre for HPC, Big Data and Quantum Computing - Proposal code CN1 CN00000013, CUP I53C22000690001, and by the INdAM-GNAMPA project 2023 ``Teoremi Limite per Dinamiche di Discesa Gradiente Stocastica: Convergenza e Generalizzazione''. 

\noindent{\bf Availability of data and materials:} Not applicable.

\def\cprime{$'$} \def\cprime{$'$}

\end{document}